\documentclass[12pt]{article}
\usepackage{lipsum}
\usepackage{pifont}
\usepackage{pinlabel}
\usepackage{rotating}
\usepackage{fancyhdr}
\usepackage[all]{xy}
\usepackage{amsmath, amsthm, latexsym, amssymb, graphicx,color,cite,mathabx,enumerate,mathrsfs}
\usepackage[left=2.0 cm,right=2.0cm,top=3cm,bottom=3cm]{geometry}

\usepackage{pinlabel} 

\usepackage{dsfont}
\usepackage{authblk}
\usepackage{thmtools}
\usepackage{thm-restate}

\usepackage{hyperref}

\usepackage{cleveref}

\declaretheorem[name=Theorem,numberwithin=section]{thm}

\newtheorem{theorem}{Theorem}[section]
\newtheorem{lemma}[theorem]{Lemma}

\newtheorem{proposition}[theorem]{Proposition}

\theoremstyle{definition}
\newtheorem{remark}[theorem]{Remark}

\newtheorem{definition}[theorem]{Definition}

\newtheorem{question}[theorem]{Question}
\newtheorem{example}[theorem]{Example}

\newcommand{\C}{\mathcal{C}}
\newcommand{\T}{\overline{T}}

\newcommand{\Pcal}{\mathcal{P}}
\newcommand{\Gcal}{\mathcal{G}}
\newcommand{\Lcal}{\mathcal{L}}

\newcommand{\mc}[1]{\mathcal{#1}}

\usepackage{amsmath}
\usepackage[svgnames]{xcolor}
\usepackage[framemethod=tikz]{mdframed}

\title{Bounds for Kirby-Thompson invariants of knotted surfaces}

\author{Rom\'an Aranda, Puttipong Pongtanapaisan, and Suixin (Cindy) Zhang}

\begin{document}
\pagenumbering{arabic}
\maketitle
\begin{abstract}
We provide sharp lower bounds for two versions of the Kirby-Thompson invariants for knotted surfaces, one of which was originally defined by Blair, Campisi, Taylor, and Tomova. The second version introduced in this paper measures distances in the dual curve complex instead of the pants complex. 
We compute the exact values of both KT-invariants for infinitely many knotted surfaces with bridge number  at most six.

\end{abstract}

\section{Introduction}
Any knot $K$ in $S^3$ can be decomposed into two trivial tangles $(S^3,K) = (B_1,T_1)\cup(B_2,T_2)$. 
Such a decomposition is called a \textbf{bridge splitting} and has been utilized to prove various important results in low-dimensional topology. For instance, Bachman and Schleimer showed in \cite{bachman2005distance} that if the distance in the curve complex between the disk sets of the two trivial tangles is sufficiently large, then the knot is hyperbolic. Instead of using the curve complex, Zupan later defined knot invariants in terms of the distance in the dual curve complex and the pants complex from disk sets of two trivial tangles in a bridge splitting \cite{zupan2013bridge}. Among other things, Zupan show that his integer-valued complexity detects the unknot and can be used to bound the hyperbolic volume of some families of knots.

Moving up one dimension, Meier and Zupan proved that any embedded closed surface $F$ in $S^4$ can be decomposed into three collections of disks \[(S^4,F)=(B^4, \mc D_1)\cup(B^4, \mc D_2)\cup(B^4, \mc D_3)\] 
satisfying certain conditions  
\cite{meier2017bridge}. 
This is called a $(b;c_1,c_2,c_3)$-\textbf{bridge trisection} where $c_i=|\mc D_i|$ and $b=|\mc D_1 \cap \mc D_2|$. Inspired by \cite{kirby2018new,zupan2013bridge}, Blair, Campisi, Taylor, and Tomova introduced the $\Lcal$-invariant for bridge trisections of knotted surfaces in the $S^4$. 
They showed that this number detects 
unlinks of unknotted surfaces.  
Another interesting aspect of $\mathcal{L}$ is that it is, a priori, not extracted from the fundamental group of the knotted surface complement, which may offer new insights. 
The $\Lcal$-invariant of a bridge trisection is roughly a sum of the lengths of certain paths in the pants complex of the surface. We introduce an $\Lcal^*$-invariant, analogous to \cite{zupan2013bridge} one dimension lower, which measures the knotting complexities using the dual curve complex rather than the pants complex. 
This work studies bridge trisections with `small' $\Lcal$- or $\Lcal^*$-invariants. 

Bridge trisections of knotted surfaces with $b\leq 3$ have been classified 
in \cite{meier2017bridge}. 
A bridge trisection is called \textbf{completely decomposable} if it is a disjoint union of perturbations of one-bridge and two-bridge trisections. In particular, the underlying surface of a completely decomposable bridge trisection is an unlink of unknotted surfaces. It was shown in \cite{meier2017bridge} that $(b;b,c_2,c_3)$-bridge trisections are completely decomposable. The same was proven for $(b;b-1,c_2,c_3)$-bridge trisections by Joseph, Meier, Miller, and Zupan in \cite{joseph2021bridge}. A result of \cite{blair2020kirby} states that if a knotted surface has $\Lcal$-invariant equal to zero, then the bridge trisection is completely decomposable. 
One of the main theorems of this work is an improved version of this. 
Theorem \ref{thm_unlinkdetector} was proven first by Ogawa in \cite{ogawa2021trisections} for the $\Lcal$-invariant of a closed 4-manifolds. 
\newtheorem*{thm:ogawa}{Theorem \ref{thm_unlinkdetector}}
\begin{thm:ogawa}
Let $\mc{T}$ be a bridge trisection with $\Lcal^*(\mc{T})\leq 2$. Then $\mc{T}$ is a distant sum of connected sums of stabilizations of one-bridge and two-bridge trisections. In particular 
the underlying surface-link is an unlink unknotted 2-spheres and unknotted nonorientable surfaces.
\end{thm:ogawa}

The first lower bounds for $\Lcal$-invariants were given in \cite{blair2020kirby} in terms of the Euler characteristic of the surface. 
Recently, the authors of this work, along with Scott Taylor, gave sharp lower bounds for $\Lcal$-invariants of  $(4;2)$-bridge trisections \cite{aranda2021bounding}. This work further explores the combinatorics of $(b;c)$-bridge trisections to give sharp lower bounds for $\Lcal$ and $\Lcal^*$. We summarize them below. 
\begin{thm}\label{thm:summary}
Let $\mc{T}$ be a $(b;c)$-bridge trisection for an embedded surface $F\subset S^4$. Suppose $\mc{T}$ is irreducible and unstabilized. Then 
\begin{equation}\label{ineq_1}
\Lcal(\mc T)\geq \Lcal^*(\mc{T})\geq 3(b+c-3). 
\end{equation}
Moreover, the following statements hold.
\begin{enumerate}[(i)]
\item If $F$ is disconnected and $c=2$, then $\Lcal^*(\mc T)\geq 3b$. 
\item If $b=4$ and $c=2$, then $\Lcal^*(\mc{T})\geq 12$. 
\item If $c\geq 2$, then $\Lcal(\mc{T})\geq 3(b+c-2)$. 
\item If $c=2$, then $\Lcal(\mc{T})\geq 3(b+1)$. 
\end{enumerate}
\end{thm}
\begin{proof}
Inequality \eqref{ineq_1} is proven in Theorem \ref{thm_lower_L*}. Inequalities (i)-(iv) are proven in theorems \ref{thm:L*_disco}, \ref{thm:L*_disco_42}, \ref{thm_lower_L}, and \ref{thm_lower_L_c2},  respectively. 
\end{proof}
The inequalities in \eqref{ineq_1}, \textit{(i)}, \textit{(ii)}, and \textit{(iv)} are sharp due to Examples \ref{rmrk:unknotted_torus}, \ref{example:T2spin_L*}, \ref{exam:L*_2b}, \ref{example:spin_toruslink_L}, respectively. 
An interesting consequence of \textit{(ii)} is Example \ref{example:yoshi_L*} which shows that the $\Lcal^*$-invariant of the 2-twist spun trefoil is 12. 
By setting $b=4$ in \textit{(iv)} we recover Theorem 3.15 of \cite{aranda2021bounding}. 
Theorems \ref{thm:summary} and \ref{thm_unlinkdetector} can be used to conclude that bridge trisections of knotted spheres are stabilized. In particular, they might be useful when trying to show that potential exotically unknotted sphere are smoothly unknotted.

In \cite{blair2020kirby} and \cite{meier2020filling}, the notion of a bridge trisection for properly embedded surfaces in $B^4$ was introduced. 
The authors of \cite{blair2020kirby} also defined the $\Lcal$-invariant of a bridge trisection of $F\subset B^4$ and gave a lower bound in terms of the Euler characteristic. We give improved inequalities for $\Lcal^*$.
\newtheorem*{thm:lower_rel}{Theorem \ref{thm:lower_relcase}}
\begin{thm:lower_rel}
Let $\mc T$ be an irreducible $(b',c_1,c_2,c_3;v)$-bridge trisection for a properly embedded surface $F$ in $B^4$. 
\begin{enumerate}[(i)]
\item If $2c_i +v \geq 2$ $\forall i$, then $\Lcal^*(\mc T) \geq 2b(\mc T) -\chi(F)$. 
\item If $2c_i +v < 2$ $\forall i$, then $\Lcal^*(\mc T) \geq 4b(\mc T) + \chi(F) -6$.
\end{enumerate}
\end{thm:lower_rel}

This work is organized as follows. Section 2 briefly reviews the definitions of bridge trisections and spinning operations for knots in $S^3$. In particular, we described how to obtain a bridge trisection for the knotted torus obtained by spinning a knot $K$ in $S^3$ with respect to a 2-sphere away from $K$ (Lemma \ref{lem:T2_triplane}). Section 3 introduces the $\Lcal^*$-invariant of a knotted surface in $S^4$ and proves Theorem \ref{thm_unlinkdetector}. 
Section 4 develops the technical lemmas used to prove the lower bounds in Theorem \ref{thm:summary}. This section is purely combinatorial and discusses properties of certain pants decompositions associated to a given trivial tangle. Sections 5 and 6 show the lower bounds for $\Lcal^*(\mc T)$ and $\Lcal(\mc T)$, respectively. At the end of each section, one can find examples computing the invariants of specific families knotted surfaces. Section 7 discusses the $\Lcal^*$-invariant for knotted surfaces with boundary. In addition to proving Theorem \ref{thm:lower_relcase}, we give a preliminary estimate for the $\Lcal^*$-invariant of the standard ribbon disk for the square knot.    

\subsection*{Acknowledgments}
We are grateful to Scott Taylor for his constant advice. PP acknowledges the Pacific Institute for the Mathematical Sciences for the support. Some of this work was done when RA visited Colby College in Spring 2022.


\section{Knotted surfaces in $S^4$}\label{section:Background}
We begin by recalling the definitions of bridge trisections of surfaces in 4-space. For the full treatment, please consult \cite{meier2017bridge}. 
Given a trivial tangle $T$, the mirror image of $T$ will be denoted by $\T$. 
A \textbf{trivial} tangle is a pair $(B,T)$ where $B$ is a 3-dimensional ball and $T$ is a collection of properly embedded arcs in $B$ such that, relative to $\partial T$, the arcs can be isotoped to $\partial B$. We think of $\partial B$ as sphere with punctures given by $\partial T$. Embedded arcs in $\partial B$ isotopic relative to the boundary to arcs of $T$ are called \textbf{shadows} of $T$. 
A \textbf{bridge splitting} of a link $L$ in $S^3$ is a decomposition $(S^3,L)=(B_1, T_1)\cup_\Sigma (B_2, T_2)$ where each $(B_i,T_i)$ is a trivial tangle. The surface $\Sigma$ is called a bridge surface for $L$. 
A \textbf{trivial disk system} is a pair $(X,\mathcal{D}),$ where $X$ is the 4-ball, and $\mathcal{D}$ is a collection of properly embedded disks simultaneously isotopic, relative to the boundary, into $\partial X$. 

\begin{definition}
A \textbf{$(b;c_1,c_2,c_3)$-bridge trisection} of an embedded closed surface in $S^4$ is a decomposition $(S^4,F) = (X_1,\mathcal{D}_1)\cup (X_2,\mathcal{D}_2)\cup (X_3,\mathcal{D}_3)$, where for $i\in \mathbb{Z}/3\mathbb{Z}$,
\begin{enumerate}
    \item $(X_i,\mathcal{D}_i)$ is a trivial disk system with $|\mc{D}_i|=c_i$,
    \item $(B_{i},T_{i}) =(X_i,\mathcal{D}_i)\cap (X_{i+1},\mathcal{D}_{i+1})$ is a $b$-string trivial tangle, and
    \item $(X_1,\mathcal{D}_1)\cap (X_2,\mathcal{D}_2)\cap (X_3,\mathcal{D}_3)$ is a sphere with $2b$ marked points. 
\end{enumerate}
\end{definition}

It follows from the definition that, for each $i$, the link $L_i = T_i \cup \T_{i+1}$ is a $c_i$-component unlink in $b$-bridge position. One of the results of \cite{meier2017bridge} is that the tuple of tangles $(T_1, T_2,T_3)$ determines the embedding of the bridge trisected surface. 
For simplicity, we refer to $(b;c,c,c)$-bridge trisections as $(b,c)$-bridge trisections. The integer $b$ is called the \textbf{bridge number} of $\mc{T}$ and $b(F)$ is the minimum $b(\mc{T})$ over all bridge trisections of $F\subset S^4$. 

\subsection{New bridge trisections from old}

Let $\mc{T}_1$ and $\mc{T}_2$ be two bridge trisections for closed surfaces $F_1$ and $F_2$ in $S^4$. Following \cite{blair2020kirby}, one can build new trisections for new knotted surfaces as follows. The \textbf{distant sum} of $\mc{T}_1$ and $\mc{T}_2$ is the new bridge trisection obtained by taking the connected sum of the two copies of $S^4$ along a neighborhood of a point in the bridge surfaces of $\mc{T}_1$ and $\mc{T}_2$ that is disjoint from the punctures. The \textbf{connected sum} of $\mc{T}_1$ and $\mc{T}_2$ is the trisection obtained by taking the connected sum of the two copies of $S^4$ along a puncture on each bridge trisection surface. 
The former is a bridge trisection for the disjoint union $F_1\sqcup F_2$, and the latter for the connected sum $F_1\# F_2$. A bridge trisection that splits in at least one of these ways is called \textbf{reducible}. Equivalently (see \cite{blair2020kirby}), $\mc{T}$ is reducible if and only if there is a simple closed curve $r$ in $\Sigma$ bounding $c$-disks on each tangle $T_i$. A bridge trisection is \textbf{irreducible} if it is not reducible. 

In dimension three, one can obtain new bridge splittings from old ones of the same link via perturbations. Loosely speaking, a perturbation of a bridge splitting of $L$ is to consider an arc of $L$ near a puncture $(0,0,0)$ parametrized by $L(t)=(0,0,t)$ and replace it with $L'(t)=(0,0,t^3-t)$. 
The bridge trisection analogue of a perturbation is called an \textbf{stabilization} and was introduced by Meier and Zupan in \cite{meier2017bridge}. They showed that any two bridge trisections of the same knotted surface in $S^4$ are related by a sequence of stabilizations and destabilizations. 
Although a stabilization does not change the underlying surface, it increases the bridge number by one and the quantities $c_i$ in a controlled way. 
For the work in this paper, we will be interested on knowing when $b(\mc{T})$ is not minimal. Hence, a criterion for when a $\mc{T}$ is stabilized is useful. 
The following result follows from Lemma 6.2 of \cite{meier2017bridge}. 

\begin{lemma}[Stabilization Criterion \cite{meier2017bridge} Rephrased]\label{lem_2}
Let $\mc{T}$ be a bridge trisection with bridge surface $\Sigma$. Let $\gamma\subset \Sigma$ be a reducing curve for $T_i \cup \T_j$ and let $a$ be a shadow for an arc in $T_k$. Suppose $|a\cap \gamma|=1$, then $\mc{T}$ is stabilized. 
\end{lemma}


\subsection{Families of knotted surfaces}


%
We review the spinning construction due to Artin \cite{artin1925isotopie}. Let $L$ be a link in $S^3$ and let $(B, L^\circ)$ be the 1-string tangle obtained from removing an open ball centered on a point of $L$. Label the points of $\partial L^\circ$ by $\{n,s\}$. The spin of $L$ is the surface $S(L)$, 
\[ \left( S^4, S(L)\right) = \left( (B,L^\circ)\times S^1\right) \cup \left( (S^2,\{n,s\})\times D^2\right).\]
If $L$ is a knot, $S(L)$ is a knotted sphere. If $L$ is disconnected, then $S(L)$ will be a link of $|L|-1$ tori and one 2-sphere. In particular, if $L$ is a 2-bridge link, then each component of $L$ is unknotted and $S(L)$ is a link of an unknotted torus and an unknotted 2-sphere. If $L$ is a $(2,n)$-torus link, then $S(L)$ is independent of which arc is used in the spinning construction. Even though this surface link has unknotted components, it is not an unlink unless $L$ is a 2-bridge unlink. 

\begin{proposition}\label{prop:spun2b}
Let $F$ be the spun of a nontrivial 2-bridge link (or knot) $L$. Then, 
\begin{enumerate}[(i)]
\item if $L$ is disconnected, then $F$ is not split,
\item the bridge number of $F$ is four, and
\item any 4-bridge trisections of $F$ is irreducible.
\end{enumerate}
\end{proposition}

\begin{proof}
(i) Via Van-Kampen’s theorem, we see that if the link $F$ splits, then the fundamental group of its complement is a rank 2 free group. Since spinning preserves the knot group, spinning a nontrivial 2-bridge link will not produce a rank 2 free group unless $L$ is an unlink.\\ 
(ii) The proof of Theorem 2.5 of \cite{aranda2021bounding} works on the disconnected case as well since $\chi(F)=2$. \\
(iii) We know that $F$ does not split from (i). Suppose that a bridge trisection $\mathcal{T}$ for a disconnected surface $F$ decomposes as $\mathcal{T}_1 \# \mathcal{T}_2$. Observe that no $\mathcal{T}_i$ can be a (2,1)-bridge trisection because that would imply that $F$ is non-orientable. Without loss of generality, this means that $\mathcal{T}_1$ is a $(3,1)$-bridge trisection and $\mathcal{T}_2$ is a $(2,2)$-bridge trisection. But a knotted surface admitting a $(2,2)$-bridge trisection is the trivial $S^2$-link. We now get a contradiction since a connected sum of a trivial $S^2$-link and an unknotted torus is a split knotted surface.
\end{proof}


Let $L$ be a link in $S^3$. Instead of removing an arc from $L$, we can perform Artin's spin construction to the whole link. More precisely, let $B\subset S^3$ be a 3-ball containing $L$ in its interior. Define the $T^2$-spin of $L$ to be the surface $T(L)$, 
\[ \left( S^4, T(L)\right) = \left( (B,L)\times S^1\right) \cup \left( (S^2,\emptyset)\times D^2\right).\]
If $L$ is disconnected, then $T(L)$ is a link of tori. If $L$ is a 2-bridge link, $T(L)$ is a link of unknotted tori in $S^4$. 
Using a similar argument as in Section 5.1 of \cite{meier2017bridge}, we can get a triplane diagram for $T(L)$ as shown in Figure \ref{fig:bandedlink}. We summarize the process here for completeness. 
\begin{figure}[h]
\labellist \small\hair 2pt  
\pinlabel {(a)}  at 25 850 
\pinlabel {(b)}  at 575 850 
\pinlabel {(c)}  at  1150 850
\endlabellist \centering
\includegraphics[width=10cm]{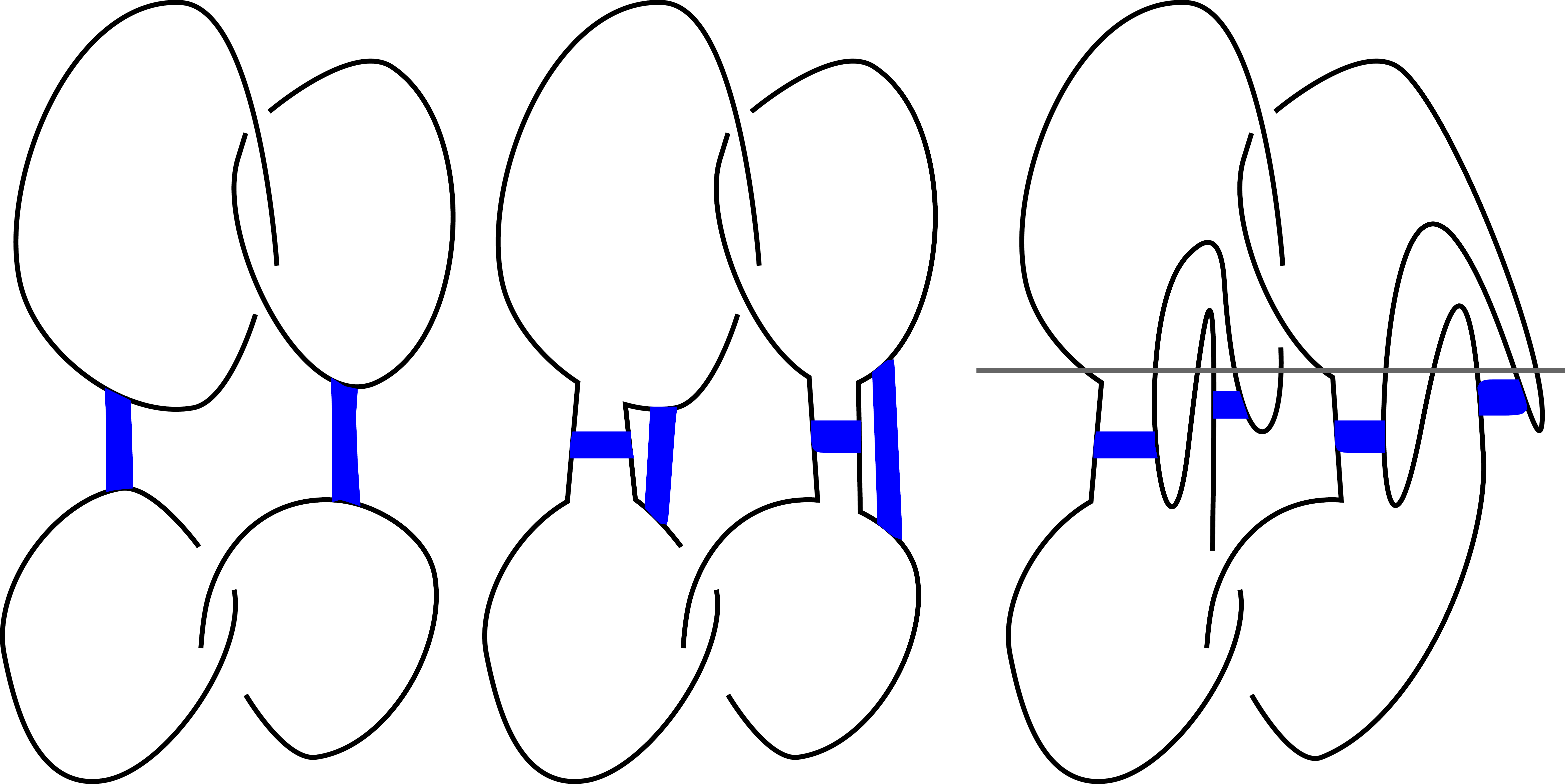}
\caption{Getting a $(3n;n)$-bridge trisection for $T(L)$.}
\label{fig:bandedlink}
\end{figure}
\begin{lemma}\label{lem:T2_triplane}
Let $L$ be a $n$-bridge link (or knot) in $S^3$. The $T^2$-spin of $L$ admits a $(3n,n)$-bridge trisection.
\end{lemma}
\begin{proof}
Begin by picking a $n$-bridge splitting for $L$. The procedure is depicted in Figure \ref{fig:bandedlink} for the Hopf link (n=2). One can represent the ribbon annuli in $B^4$ bounded by $L\sqcup \overline{L}$ as a split union of $L$ and its mirror image connected by n bands, one per bridge of $L$. 
The bands indicate the saddles for the annuli. A banded unlink diagram for the double of this annuli can be obtained by band surgeries and dual band attachments. Thus we get a banded presentation for $T(L)$ with $n$ minima, $2n$ saddles and $n$ maxima. Isotope the banded unlink diagram as in Figure \ref{fig:bandedlink}(c) (once per pair of saddles). 
The isotopy was done so that the bands are parallel to $y^*$ embedded arcs in the bridge surface and there is a collection of shadows $a^*$ for the bottom tangle of the unlink in $3n$-bridge position such that $a^*\cup y^*$ is a collection of pairwise disjoint arcs. The bridge splitting of the $3n$-unlink, together with the bands satisfying the previous properties is, by definition a banded bridge splitting \cite{meier2017bridge}. By Lemma 3.2 of \cite{meier2017bridge}, this yields a $(3n,n)$-bridge trisection of $T(L)$ as in Figure \ref{fig_T2_spun_knot_triplane}. 
We can obtain the tuple of tangles for $\mc{T}$ as follows: the bottom and top tangle in Figure \ref{fig:bandedlink}(c) correspond to $T_1$ and $T_3$, and the tangle $T_2$ is obtained by performing a band surgery to $T_3$. 
\end{proof}
\begin{figure}[h]
\centering
\includegraphics[width=1\textwidth]{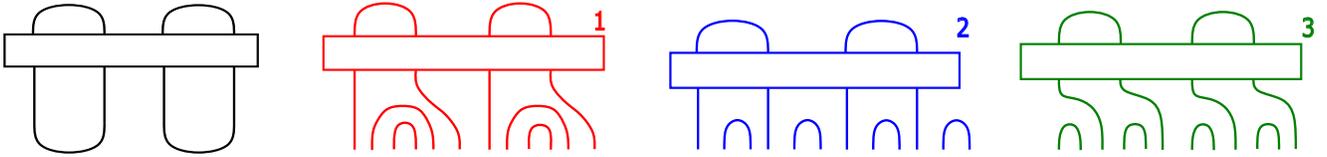}
\caption{A $(6,2)$-bridge tri-plane diagram for the $T^2$-spin of a $2$-bridge link (or knot) $L$ given in bridge position (left).}
\label{fig_T2_spun_knot_triplane}
\end{figure}
\begin{lemma}\label{lem:6bridge}
Let $L$ be a nontrivial 2-bridge link. The bridge number of $T(L)$ is six. Moreover, every minimal $(b;c_1,c_2,c_3)$-bridge trisection for $T(L)$ satisfies $(b;c_1,c_2,c_3)=(6;2,2,2)$.
\end{lemma}
\begin{proof}
This follows from the fact that the bridge number of the unknotted torus is equal to three and that if $\mc{T}$ is a $(3;c_1,c_2,c_3)$-bridge trisection for it, then $c_i=1$. 
\end{proof}
\begin{lemma}\label{lem:bridget2}
Let $L=L_1\cup L_2$ be a nontrivial 2-bridge 2-component link. Then, any minimal bridge trisection of $T(L)$ is irreducible.
\end{lemma}
\begin{proof}
Suppose that a trisection $\mathcal{T}$ of $T(L)$ decomposes as a connected sum $\mathcal{T}_1\#\mathcal{T}_2.$ Without loss of generality, due to Lemma \ref{lem:6bridge} and the fact that $T(L)$ is disconnected, $\mathcal{T}_1$ admits a $(4,2)$-bridge trisection and $\mathcal{T}_2$ admits a $(3,1)$-bridge trisection. In other words, $T(L)$ is the result of connect summing the unknotted torus $U$ with a 2-component knotted surface, where one component $F_1$ is the sphere and the other component $F_2$ is the torus.

We now use an argument similar to Livingston's proof of Corollary 3.2 in \cite{livingston1985stably}. Namely, let $\alpha$ be an element on the torus of $\pi_1(F_1\# U)$ and denote the image of $\alpha$ induced by inclusion $\iota:\pi_1(F_1\#U)\rightarrow \pi_1(S^4\backslash T(L))$ by $\tilde{\alpha}$. On one hand, any such $\tilde{\alpha}=1$ because $U$ is unknotted. On the other hand, since $T(L)$ arises from spinning, we claim that there is a nontrivial element $\tilde{m}\neq 1\in \pi_1(S^4\backslash T(L))$. Looking back to the definition of $(S^4, T(L)) = ( (B,L)\times S^1) \cup ((S^2,\emptyset)\times D^2),$ we let $m=L_1 \times \lbrace \text{point} \rbrace \subset L_1\times S^1$. Recall that $\pi_1(S^4\backslash T(L))$ is isomorphic to $\pi_1(S^3\backslash L)$ and if $\tilde{m}=1$ in $\pi_1(S^4\backslash T(L))$ then $\tilde{m}=1$ in $\pi_1(S^3\backslash L)$. This would imply that the map of $\pi_1(\partial(\nu(L_1))\rightarrow\pi_1(S^3\backslash L)$ is not injective. By Dehn's Lemma, $L_1$ bounds a disk in $S^3\backslash L$ contradicting the fact that $L$ is a non-split 2-bridge link. 

\end{proof}

\section{Kirby-Thompson invariants}

A \textbf{pair-of-pants} is a connected planar surface (possibly with punctures and boundaries) with Euler characteristic equal to -1.
Let $\Sigma$ be a sphere 
with punctures. A \textbf{pants decomposition} is a collection $P$ of essential simple closed curves in $\Sigma$ such that $\Sigma-P$ is a disjoint union of pairs of pants. Every pants decomposition of $\Sigma$ has $2b - 3$ 
non-isotopic curves. 
The \textbf{dual curve graph} $C^*(\Sigma)$ is a graph, where each vertex is pants decomposition of $\Sigma$.
Two vertices are connected by an edge if they differ by a single curve. More specifically, if $u'$ is obtained from $u$ by through the process of removing one curve and replacing it with a curve disjoint from all the other curves so that the new curves still cut $\Sigma$ into pairs of pants. 
An important subgraph of $C^*(\Sigma)$ is 
the \textbf{pants complex} $\Pcal(\Sigma)$, where 
two vertices $u$ and $u'$ are connected by an edge if $u'$ can be obtained from $u$ by removing a curve $c$ and replacing $c$ with a curve $c'$, where $|c\cap c'|=2$. We will refer to edges in $\Pcal(\Sigma)$ as A-moves. 
Examples of edges in both complexes can be found in Figure \ref{fig:dualtrefoilpath}.

We say that a properly embedded disk $D$ in a trivial tangle $(B,T)$ is a \textbf{compressing disk} (resp. \textbf{cut-disk}) if $D$ is disjoint from $T$ (resp. intersects $T$ once in its interior) and its boundary is essential in $\Sigma$. A \textbf{c-disk} is by definition either a compressing disk or a cut-disk. We say that a vertex $v \in C^*(\Sigma)$ or $P(\Sigma)$ belongs to the \textbf{disk set} $\mathcal{P}_c(T)$ if every curve of $u$ bounds a $c$-disk in $T$.
Let $L$ be a link in $S^3$ with a bridge surface $\Sigma$; i.e. $(S^3,L)=(B_1,T_1)\cup_\Sigma (B_2,T_2)$. 
An essential simple closed curve in $\Sigma$ is \textbf{reducing} (resp. \textbf{cut-reducing}) if it bounds a compressing disk (resp. cut-disk) in both tangles. 

\begin{definition}\label{definition:efficientpair}
Let $(B_1,T_1)$ and $(B_2,T_2)$ be two trivial tangles with the same boundary surface $\Sigma$. 
A pair of vertices $(u ,w) \in \mathcal{P}_c(T_1) \times \mathcal{P}_c(T_2)$ 
is said to be an \textbf{efficient defining pair for $T_1\cup \T_2$} if the distance $d\left(\mathcal{P}_c(T_1), \mathcal{P}_c(T_1)\right) =d(u,w)$ in $\C^*(\Sigma)$. 
\end{definition}
From the work of Hatcher and Thurston \cite{hatcher_thurston}, the distance in the definition above is always finite. 
In \cite{zupan2013bridge}, Zupan showed that the distance $d(u,v)$ of an efficient defining pair $(u,v)$ for a $c$-component unlink in $b$-bridge position is $b-c$. Blair, Campisi, Taylor and Tomova \cite{blair2020kirby} further studied efficient defining pairs for unlinks and used them to define the $\Lcal$-invariant of embedded surfaces in $S^4$ or $B^4$. Using this philosophy, we now introduce the $\Lcal^*$-invariant. 
As in \cite{blair2020kirby}, one can talk about $\Lcal^*$-invariants of a properly embedded surface in $B^4$. In order to simplify the presentation, we defer the study of the relative case until Section \ref{section:relative}. 
A surface $\Sigma$ is \textbf{admissible} if $\chi(\Sigma)\leq -2$.
\begin{definition}\label{definition:star}
Let $F$ be a closed surface embedded in $S^4$.  
Let $\mc{T}$ be a bridge trisection for $F$. 
If the bridge surface of $\mc{T}$ 
is not admissible, then we define $\mathcal{L}^*(\mc{T}) = 0$. 
For each $i<j$, let $\left(P^i_{ij}, P^j_{ij}\right)$ be an efficient pair for the unlink  $T_i\cup \T_j$. 
We define $\mathcal{L}^*(\mathcal{T})$ 
to be the minimum of 
\[
d(P^1_{12},P^1_{31})+d(P^2_{23},P^2_{23})+d(P^3_{31},P^3_{23}),
\]
over all choices of efficient defining pairs. The surface invariant \textbf{$\mathcal{L}^*(F)$} is defined to be the minimum $\mathcal{L}^*(\mathcal{T})$ over all bridge trisections $\mathcal{T}$ for $F$ such that $b(\mathcal{T})=b(F)$.
\end{definition}

\begin{remark}
If one replaces the distance $d(x,y)$ in $C^*(\Sigma)$ with the pants complex distance $d_{\Pcal(\Sigma)}(x,y)$ in definitions \ref{definition:efficientpair} and \ref{definition:star}, we recover the original Kirby-Thompson invariant $\mathcal{L}(F)$ defined in \cite{blair2020kirby}. In particular, $\Lcal(\mc T)\geq \Lcal^* (\mc T)$. This paper will give lower bounds for both invariants in the closed setting.  
\end{remark}

\begin{figure}[h!] 
\labellist \small\hair 2pt  
\pinlabel {$T_{1}$}  at 100 295 
\pinlabel {{\color{white}$T_{2}$}}  at 450 505 
\pinlabel {{\color{white}$T_{3}$}}  at 510 105 
\pinlabel \scriptsize$P_{12}^1$  at 166 376 
\pinlabel \scriptsize$P_{31}^1$  at 162 200 
\pinlabel \scriptsize$P_{31}^3$  at 386 90 
\pinlabel \scriptsize$P_{12}^2$  at 356 490 
\pinlabel \scriptsize$P_{23}^2$  at 519 401 
\pinlabel \scriptsize$P_{23}^3$  at 522 216 
\pinlabel {
\begin{rotate}{-24}
{\scriptsize
$b(\mathcal{T})-c_1$
}
\end{rotate}
}
at 170 70
\pinlabel {
\begin{rotate}{30}
{\scriptsize
$b(\mathcal{T})-c_2$
}
\end{rotate}
}
at 150 480
\pinlabel 
\scriptsize
$b(\mathcal{T})-c_3$
at 700 295 
\endlabellist \centering 
\includegraphics[width=0.4\textwidth]{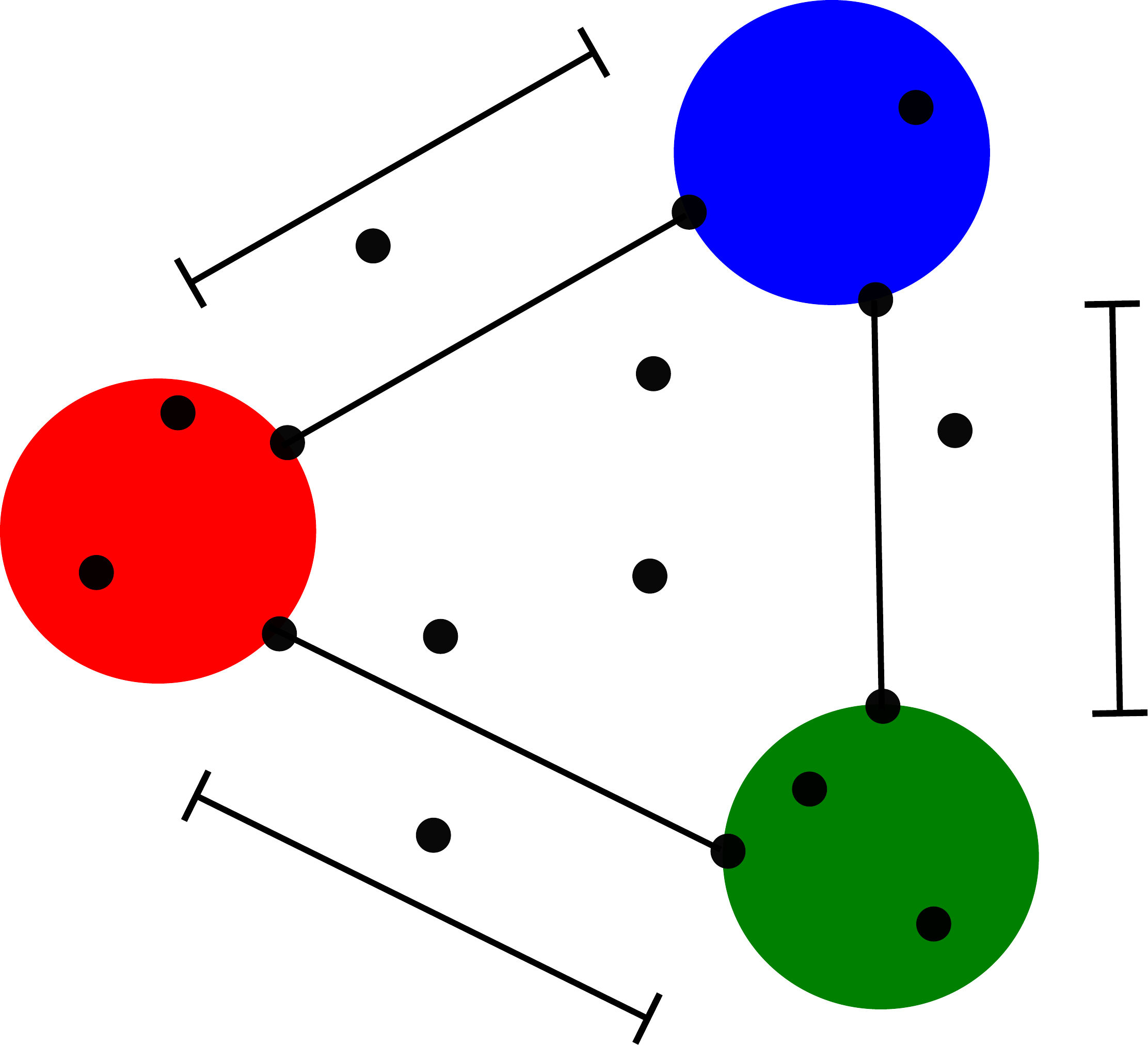}
\caption{Defining $\mathcal{L}^* (\mc T)$ via efficient defining pairs. The ellipses represent the disk sets. The line joining $P_{ij}^{i}$ to $P_{ij}^{j}$ represents a geodesic path in the pants complex.}
\label{fig_L_invariant}
\end{figure}
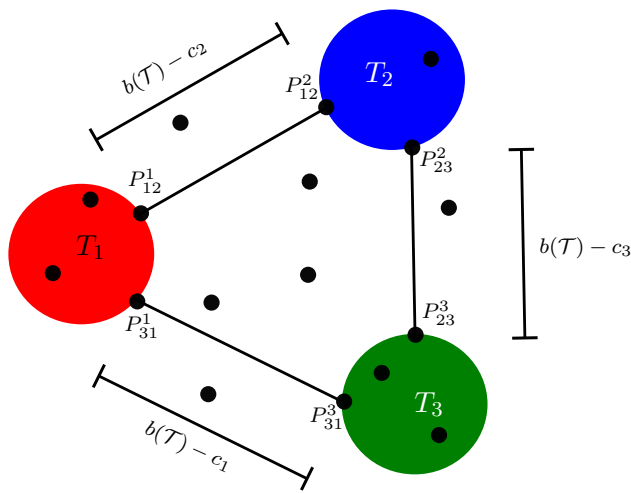



We finish this section with a result about the structure of efficient defining pairs for the unlink. It is important to mention that the proof of Lemma \ref{lem_1} appears in Lemma 5.6 of \cite{blair2020kirby}. One can extract more detailed conclusions from the original proof in \cite{blair2020kirby} which are key for the work in this paper. For completeness, we include a proof-sketch. 
A first application of Lemma \ref{lem_1} is Theorem \ref{thm_unlinkdetector} which implies that bridge trisections with small $\Lcal^*$-invariant are standard. 

\begin{lemma}[Lemma 5.6 of \cite{blair2020kirby}]\label{lem_1}
Let $(P,P')$ be a geodesic pair for a {$c$-component} unlink in {$b$-bridge} position with bridge surface $\Sigma$. There exist collections of simple closed curves $\psi$, $g$, $g'$ in $\Sigma$ with $P=\psi \cup g$ and $P'=\psi \cup g'$ satisfying the following properties, 
\begin{enumerate}[(i)]
\item the sets have size $|g|=|g'|=b-c$ and $|\psi|=b+c-3$, 
\item any geodesic in 
$C^*(\Sigma)$ connecting $P$ and $P'$ moves each curve in $g$ exactly once,  
\item for all $x\in g$, $y\in g'$, $|x\cap y|\leq 2$, and 
\item for all $x\in g$, there is a unique $x'\in g'$ such that $|x\cap x'|=2$.
\end{enumerate}
In particular, the collections of geodesic pairs in $C^*(\Sigma)$ and $\Pcal(\Sigma)$ are the same. 
\end{lemma}
\begin{proof}
 We will induct on $b$. If $b=2,$ then $v$ and $w$ each consists of exactly one compressing disk. There are two possibilities. If the curve in $v$ is the same curve as the one in $w$, then there is an essential separating sphere coming from the gluing the compressing disks on two sides and $c=2$. Therefore, $d(u,w)=0= b-c.$ If the curve in $v$ is different from the curve in $w$, then $d(u,w)=1 =b-c.$ Furthermore, since $L$ is the 2-bridge unknot in this case, the two compressing disks on opposite sides must be neighborhoods of bridge disks $D_+$ and $D_-$. Here, $|D_+\cap D_-|$ in precisely one point on the knot, and so $|u\cap w|=2.$

Suppose that the statements hold for all bridge spheres with bridge number $b'< b.$ Using the fact that bridge positions for unlinks are perturbed \cite{otal1982perturbed}, the authors in \cite{blair2020kirby} produced a geodesic of length $b-c$ with the three properties stated in the lemma. This means that a path without a common curve cannot be a geodesic as the length of such a path is at least $2b-3$ which exceeds the upper bound of $b-c$ provided in \cite{blair2020kirby}. 
Let $\gamma$ be such a common curve bounding the same type of compression on both sides. In other words, there is a sphere $S$ either disjoint from $L$ or $|S\cap L| =2$ such that $S\cap\Sigma=\gamma$. Surger $(S^3,L)$ along $S$ to obtain a decomposition $(S^3,L)=(S^3,L_1)\cup (S^3,L_2)$. Let $b_i$ and $c_i$ denote the bridge number and the number of components of $L_i$, respectively. Note that $(S^3,L_i)$ inherits a bridge sphere $\Sigma_i$ from $\Sigma.$ Let $u_i$ and $w_i$ be restrictions of $u$ and $w$ to $\Sigma_i.$ Then, a geodesic $l$ connecting $u$ to $w$ restricts to a geodesic $l_i$ connecting $u_i$ and $w_i.$ Using the induction hypothesis, since $b_i< b$ and $c_i<c,$ we know that the length of $l_i$ is at least $b_i-c_i.$ Therefore, $d(u,w)\geq b_1+b_2-c_1-c_2=b-c$ as desired.
\end{proof}


\begin{theorem}\label{thm_unlinkdetector}
Let $\mc{T}$ be a bridge trisection with $\Lcal^*(\mc{T})\leq 2$. Then $\mc{T}$ is a distant sum of connected sums of stabilizations of one-bridge and two-bridge trisections. In particular 
the underlying surface-link is an unlink unknotted 2-spheres and unknotted nonorientable surfaces.
\end{theorem}

\begin{proof}
Throughout this proof, when we say that a bridge trisection is standard, we mean that $\mc{T}$ is as in the conclusion of the Theorem. 
We will proceed by induction on $b\geq 1$. By \cite{meier2017bridge}, every bridge trisection with $b\leq 3$ is either standard 
or the unique bridge trisection of the unknotted torus in $S^4$. The latter satisfies $\Lcal^*(\mc{T})=3$ (Remark \ref{rmrk:unknotted_torus}). So the result holds for $b\leq 3$. If $b=4$, the condition $\Lcal^*(\mc{T})\leq 2$ implies the existence of two efficient pairs $\left( P^i_{ij}, P^j_{ij}\right)$ and $\left(P^i_{ik}, P^k_{ik}\right)$ with a common curve $x\in \left(P^i_{ij}\cap P^j_{ij}\right) \cap \left(P^i_{ik}\cap P^k_{ik}\right)$. Hence, $\mc{T}$ is the distance sum or connected sum of two bridge trisections with bridge number $b'\leq 3$. Hence, $\mc{T}$ is standard. 

Suppose that, for some $b\geq 4$, the result hols for all bridge trisections with bridge number at most $b$. Let $\mc{T}$ be a $(b+1;c_1,c_2,c_3)$-bridge trisection with $\Lcal^*(\mc{T})\leq 2$. 
By the definition of $\Lcal^*(\mc{T})$, 
there exists a shortest loop $\lambda$ in the dual complex $C^*(\Sigma)$ passing through all pants $P^i_{ij}$ in the cyclic order $P^1_{12}$, $P^2_{12}$, $P^2_{23}$, $P^3_{23}$, $P^3_{13}$, $P^1_{13}$. Since $\Lcal^*(\mc{T})\leq 2$, $\lambda$ has length at most $2+3(b+1)-(c_1+c_2+c_3)$. 
Part (ii) of Lemma \ref{lem_1} states that every efficient pair has $(b+1)+c_i-3$ common curves. since $b\geq 4$ and $c_i\geq 1$, $(b+1)+c_i-3>2$. Thus, there exists a curve $x$ that does not move along $\lambda$. In particular $x\in \bigcap_{i,j} \left(P^i_{ij}\cap P^j_{ij}\right)$ and $\mc{T}$ is reducible. Let $\mc{T}_1$ and $\mc{T}_2$ be the two new bridge trisections decomposing $\mc{T}$ either as a distance sum of connected sum. The bridge number of each $\mc{T}_i$ is at most $b$. Moreover, since $x$ is fixed by $\lambda$, we get that $\Lcal^*(\mc{T}_1) +\Lcal^*(\mc{T}_2)=\Lcal^*(\mc{T})\leq 2$ and so $\Lcal^*(\mc{T})\leq 2$. Thus, by the inductive hypothesis, each $\mc{T}_i$ is standard. 
\end{proof}



\section{Combinatorics of pants in $\Pcal_c(T)$} 
The goal of this section is to show Propositions \ref{lem_5} and \ref{lem_3} which are sufficient conditions to ensure that a loop bounds a c-disk in a tangle. Along this section, $T$ is a $b$-string trivial tangle in a 3-ball with boundary a surface $\Sigma$. By $p\sim_T q$, we mean that the punctures $p$ and $q$ are connected by an arc in the tangle $T$.

\begin{lemma}\label{lem_0}
The following hold for any planar surface. 
\begin{enumerate}
\item Any pants decomposition of an $n$-punctured disk has at least $\lfloor {n/2} \rfloor$ curves bounding an even number of punctures. 
\item Given pairwise disjoint disks bounding $x_1\geq 1$ punctures each, any pants decomposition of the union of the disks has at least $\frac{1}{2}\sum x_i - \frac{1}{2}\#\left\{x_i: x_i \text{ is odd}\right\}$ curves bounding an even number of punctures. 
\item If $P$ is a pants decomposition for $\Sigma_{0,2b}$, then $P$ has at most $b-2$ curves bounding an odd number of punctures. 
\end{enumerate}
\end{lemma}
\begin{proof}
The first claim follows by induction on $n\geq 1$. The second one from the fact that $x/2- \lfloor {x/2} \rfloor$ is $1/2$ if $x$ is odd and $0$ otherwise.  
The third claim follows from the first one by thinking of the any pants decomposition in $\Sigma_{0,2b}$ as a pants decomposition for a disk with $(2b-1)$ punctures.
\end{proof}

\subsubsection*{The graph $\Gcal(P)$} 

Let $(B,T)$ be a trivial tangle and let $P\in \Pcal_c(T)$ be a pants decomposition for $\Sigma$ so that every curve in $P$ bounds a c-disk for $T$. Define a graph $\Gcal(P)$ as follows: The vertices of $\Gcal$ are either pants components of $\Sigma-P$ or punctures of $\Sigma$. The edges of $\Gcal$ correspond to loops in $P$ and to curves $\partial \eta (p)$ for each puncture $p$. 
\begin{figure}[ht!]
\labellist \small\hair 2pt 
\pinlabel {$1$}  at 50 64  
\pinlabel {$2$}  at 230 80 
\pinlabel {$3$}  at 175 70 
\pinlabel {$4$}  at 145 59  
\pinlabel {$5$}  at 220 65  

\pinlabel {$p$}  at -5 40  
\pinlabel {$a$}  at 35 42 
\pinlabel {$b$}  at 59 45  
\pinlabel {$c$}  at 115 42 
\pinlabel {$d$}  at 136 45  
\pinlabel {$q$}  at 162 40 
\pinlabel {$s$}  at 205 44  
\pinlabel {$t$}  at 228 44  

\pinlabel {$1$}  at 367 72
\pinlabel {$2$}  at 395 84 
\pinlabel {$3$}  at 425 50 
\pinlabel {$4$}  at 405 33 
\pinlabel {$5$}  at 435 70 

\pinlabel {$p$}  at 330 52  
\pinlabel {$a$}  at 355 35  
\pinlabel {$b$}  at 392 35  
\pinlabel {$c$}  at 380 01  
\pinlabel {$d$}  at 427 01  
\pinlabel {$q$}  at 445 12  
\pinlabel {$s$}  at 445 37  
\pinlabel {$t$}  at 485 37  

\endlabellist \centering
\includegraphics[width=15cm]{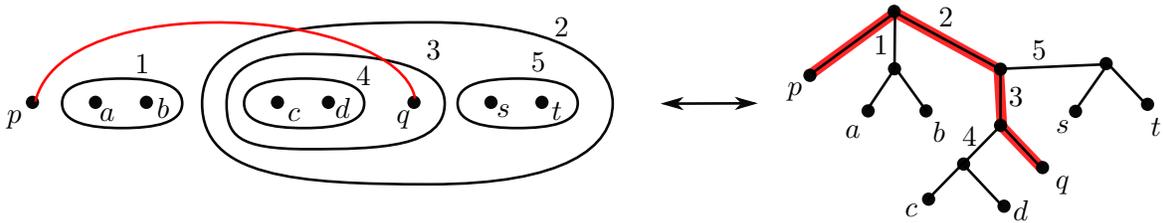}
\centering
\caption{An example of a pants decomposition $P$ and $\Gcal(P)$.}
\label{fig_G(P)}
\end{figure}

The graph $\Gcal(P)$ is a tree with vertices of degree 3 or 1. The leaves (degree 1) correspond to punctures of $\Sigma$. Two leaves having the same neighbor correspond to a 2-punctured disk component in $\Sigma-P$. For any two punctures $p$, $q$, denote by $\lambda_{p,q}\subset \Gcal(P)$ the unique path connecting $p$ and $q$. 

The following lemma encapsulates key properties of $\Gcal(P)$. 

\begin{lemma}\label{lem_4}
Let $P\in \Pcal_c(T)$, $\Gcal(P)$ as above. 
\begin{enumerate}[(i)]
\item Every vertex of $\Gcal(P)$ is visited by at most one path $\lambda_{p,q}$ where $p\sim_T q$. 
\item If $p\sim_T q$ and $r\sim_T s$ are distinct punctures, then the paths $\lambda_{p,q}$ and $\lambda_{r,s}$ are disjoint. 
\item Let $p\sim_T q$. Let $a\subset \Sigma$ be any embedded arc with $\partial a=\{p,q\}$ satisfying the condition: $\forall\alpha \in P$, $a$ intersects $\alpha$ once if $\alpha$ is an edge of $\lambda_{p,q}$, and $a \cap \alpha =\emptyset$ otherwise. Then $a$ is a shadow of an arc of $T$. 
\end{enumerate}
\end{lemma}

\begin{proof}
Part (ii) follows from (i). When proving part (iii), we will build a complete collection of shadows for $T$ using the procedure in Part (iii). This way, cut-curves $\alpha\in P$ will intersect only one shadow of $T$, and curves $\alpha \in P$ bounding compression disks will be disjoint from all shadows. In particular, every edge $\alpha \in V(\Gcal)$ belongs to at most one path $\lambda_{p,q}$ with $p\sim_T q$. Part (i) will follow since the degree of each vertex in $\Gcal$ is at most three. 

We now prove Part (iii), as stated before, we will build a complete collection of shadow arcs for $T$. 
Begin with pairs of leaves $(p,q)$ in $\Gcal(P)$ having a common neighbor $v$. The vertex $v$ corresponds to a twice punctured disk $v$ in $\Sigma-P$ (see Figure \ref{fig_outer_leaves}). Let $e_{p,q}$ be the boundary of $v$. Since $e_{p,q}\in P$, $e_{p,q}$ must bound a compressing disk for $T$ and so $p\sim_T q$. In particular, the path $\lambda_{p,q}$ traverses the edges $p$ and $q$. Moreover, any arc $a$ connecting $p$ and $q$ that built as in Part (iii) stays inside $v$. Thus $a$ is a shadow of $T$. 

\begin{figure}[ht!]
\labellist \small\hair 2pt  
\pinlabel {$x$}  at 01 47 
\pinlabel {$w$}  at 30 55 
\pinlabel {$y$}  at 60 47 
\pinlabel {$e_{p,q}$}  at 40 36
\pinlabel {$v$}  at 26 28
\pinlabel {{\color{red} $\lambda_{p,q}$}}  at 50 20
\pinlabel {$p$}  at 7 5
\pinlabel {$q$}  at 60 5

\pinlabel {$e_{p,q}$}  at 130 55
\pinlabel {{\color{red} $a$}}  at 160 43 
\pinlabel {$p$} at 134 26
\pinlabel {$q$} at 182 26
\pinlabel {$v$} at 160 25

\endlabellist \centering
\includegraphics[width=8cm]{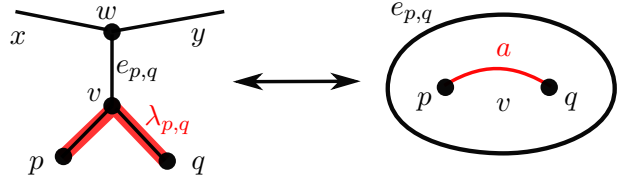}
\centering
\caption{Outer leaves in $\Gcal(P)$ correspond to twice punctured disks.}
\label{fig_outer_leaves}
\end{figure}

Cut $B'$ open along the disk bounded by $e_{p,q}$ to obtain two 3-balls with trivial trivial tangles. Let $(B', T')$ be the $(n-1)$-tangle that does not contain $p$ and $q$, and let $\Sigma'=\partial B'-T'$. The curves in $P':=P-\{e_{p,q}\}$ can be taken to be inside $\Sigma'$. 
Denote by $w$ the other endpoint of the edge $e_{p,q}$ in $\Gcal(P)$ and let $x,y$ be the other two edges with one endpoint equal to $w$ (see Figure \ref{fig_outer_leaves}). Notice that the loops $x$ and $y$ become isotopic in $\Sigma'$. By taking $x=y$, we obtain that $P'$ is a pants decomposition of $\Sigma'$. 
Furthermore, $P'\in \Pcal_c(T')$. Now observe that the graphs $\Gcal(P)$ and $\Gcal(P')$ are related in the following way: $\Gcal(P')$ is the result of removing the edged $\{p,q,e_{p,q}\}$, and replacing the path $\overline{xy}$ with the one edge $x=y$. This way any arc $a\subset \Sigma'$ satisfying the conditions in Part (iii) can be `lifted' to an arc $a\subset \Sigma$ also satisfying the conditions in Part (iii). We can then proceed recursively with $P'$ instead of $P$ to create a complete collection of shadows for $T$. 
\end{proof}


\begin{proposition}\label{lem_5}
Let $P\in \Pcal_c({T})$ and let $p\sim_T q$ be two punctures. Suppose that $\gamma$ is a simple closed curve satisfying $x\in P$, $|\gamma\cap x|\leq 2$, $\forall x\in P$. Then there exists a shadow $a$ for $T$ connecting $p$ and $q$ such that, 
\begin{enumerate}[(i)]
\item if $\gamma$ separates $p$ and $q$, then $|a\cap \gamma|=1$;
\item if $\gamma$ does not separate $p$ and $q$, then $|a\cap \gamma|=0$.
\end{enumerate}
\end{proposition}

\begin{proof}[Proof of Proposition \ref{lem_5}]
Consider the graph $\Gcal(P)$ and the path $\lambda_{p,q}$. 
Denote the edges of $\lambda_{p,q}$ by $p=x_0, x_1,\dots, x_n=q$ and the vertices by $v_i=\partial x_i \cap \partial x_{i+1}$. For each $1\leq k\leq n$, let $F_i=v_0\cup v_1 \cup \dots \cup v_k$ be the planar subsurface of $\Sigma$ traced by the first $k$ pairs of pants in the path $\lambda_{p,q}$. 
Label each edge $x_k$ of $\lambda_{p,q}$ with $p$, $q$ or $0$, depending if $x_k$ is in the side of $p$, $q$ or if $|x_k\cap \gamma|=2$. If $\gamma$ does not separate $p$ and $q$, the labels $p$ and $q$ are the same. 
Observe that $x_0=\partial\eta(p)$ and $x_n=\partial\eta(q)$ have labels $p$ and $q$, respectively. 
\\
\textbf{Claim 1.} $x_k$ separates $p$ and $x_{k+1}$.

By definition of $F_n$, $p$ and $x_{k+1}$ lie in distinct components of $F_n-x_k$. Since $F_n\subset \Sigma$ are planar surfaces, $p$ and $x_{k+1}$ lie in distinct components of $\Sigma-x_{k+1}$. \\
\textbf{Claim 2.} If $x_k$ has label $q$, then $x_{k+1}$ has label $q$.

Suppose now that $x_k$ has label $q$. In particular, $\gamma$ and $x_k$ are disjoint. 
Claim 1 implies that $q$ and $x_{k+1}$ lie in the same side of $x_k$. All this is occurring inside a planar surface, so $q$ and $x_{k+1}$ must be lie in the same side of $\gamma$. Hence, $x_{k+1}$ is labeled $q$. 

By considering the path $\overline{\lambda}$ in the opposite direction, we obtain that if $x_{k+1}$ has label $p$, then $x_k$ has label $p$. 
Thus, along $\lambda$, once the label $p$ changes (resp. $q$ appears), it does not appear again (resp. change).
Hence, there exist indices $0\leq i\leq j\leq n$ such that 
the label of $x_l$ is $p$, $0$, and $q$ for $l$ in the intervals $[0,i]$, $(i,j)$, and $[j,n]$, respectively. 
We now choose the shadow arc $a$ using the procedure in Part (iii) of Lemma \ref{lem_4}. By the lemma, such arc will be a shadow for the tangle $T$ so it is enough to guarantee the intersection number $|\gamma \cap a|$.

Suppose first that $p$ and $q$ lie in the same side of $\gamma$. The labels $p$ and $q$ are equal so, by Claim 2, all labels in $\lambda$ are equal to $p$. Let $0<l<n$ and let $y\in P$ be the simple closed curve so that $\partial v_l = y \cup x_l\cup x_{l+1}$. We know that $|\gamma\cap y|\leq 2$ and $\gamma$ is disjoint from $x_l$ and $x_{l+1}$. 
Thus, $v_l \cap \gamma$ is either empty or an essential embedded arc in $v_l$ with both boundaries in $y$ (see Figure \ref{fig_vk} (e)-(f)). The latter forces $\gamma$ to separate $p$ and $q$, so it cannot occur. Thus, every $v_l$ is disjoint from $\gamma$ and any arc $a$ built as in Lemma \ref{lem_4} will be disjoint from $\gamma$. 

Suppose now that $\gamma$ separates $p$ and $q$. We focus on the pair of pants $v_l$ for each $0\leq l<n$. 
We know that $\gamma$ intersects every curve in $\partial v_l \subset P$ at most twice. Thus, up to a surface homeomorphism, the intersection $\gamma \cap v_l$ can be depicted as in Figure \ref{fig_vk}. If $l<j-1$, the intersection pattern cannot be as in Figure \ref{fig_vk} (f). Observe that in all the other cases, we can find an arc $a$ (in blue) connecting $x_l$ and $x_{l+1}$ that is disjoint from $\gamma$. Similarly, for $j-1<l$, we can find an arc $a$ in $v_l$ disjoint from $\gamma$. We now study the pants $v_{j-1}$. Since $x_{j-1}$ and $x_j$ lie in distinct sides of $\gamma$, $v_{j-1}$ must look as in Figure \ref{fig_vk} (f). Here, there is an arc $a\subset v_{j-1}$ intersecting $\gamma$ exactly once. By gluing all the arcs discussed above, we build an arc connecting $p$ and $q$ that intersects $\gamma$ once. 
\begin{figure}[ht!]
\labellist \small\hair 2pt  

\pinlabel {(a)} at 01 125
\pinlabel {$l$} at 22 99

\pinlabel {(b)} at 115 125
\pinlabel {$l$} at 134 99

\pinlabel {(c)} at 225 125
\pinlabel {$l$} at 245 99

\pinlabel {(d)} at 01 55
\pinlabel {$l$} at 22 26

\pinlabel {(e)} at 115 55
\pinlabel {$l$} at 134 26

\pinlabel {(f)} at 225 55
\pinlabel {$l$} at 245 26

\scriptsize 
\pinlabel {$l+1$} at 183 26
\pinlabel {$l+1$} at 70 26
\pinlabel {$l+1$} at 293 26
\pinlabel {$l+1$} at 183 99
\pinlabel {$l+1$} at 70 99
\pinlabel {$l+1$} at 293 99

\endlabellist \centering
\includegraphics[width=12cm]{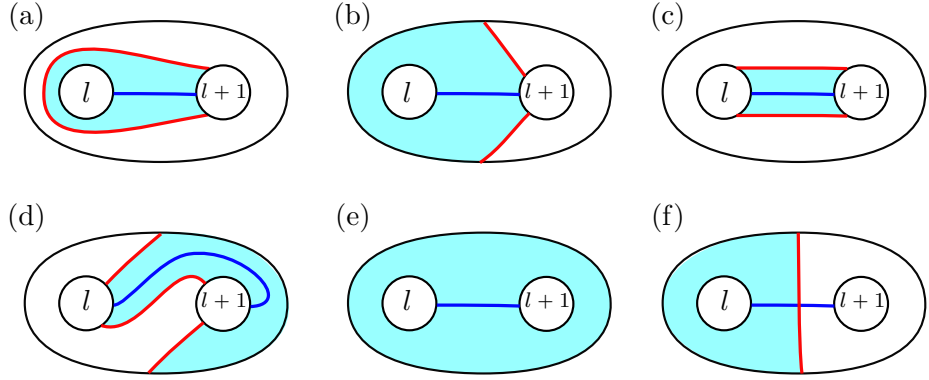}
\centering
\caption{The pair of pants $v_l$. The red arcs correspond to $\gamma \cap v_l$.}
\label{fig_vk}
\end{figure}
\end{proof}


\begin{proposition} \label{lem_3}
Let $\gamma\subset \Sigma$ be an essential simple closed curve and $P\in \Pcal_c(T)$. Suppose that $|\gamma\cap x|\leq 2$ for all $x\in P$. If every pair of punctures $p\sim_T q$ belong to the same side of $\gamma$, then $\gamma$ bounds a compressing disk for $T$. 
\end{proposition}
\begin{proof}
By Proposition \ref{lem_5}, every arc of $T$ has a shadow disjoint from $\gamma$. 
\end{proof}


\subsubsection*{Short paths in $\Pcal(\Sigma)$}

\begin{lemma}\label{lem_6}
Let $P$ be a pants decomposition for $\Sigma$ and let $x,y,\gamma\subset \Sigma$ be essential simple closed curves satisfying $|\gamma \cap z|\leq 2, \forall z\in P$. Suppose $x\in P$ and $P\mapsto (P-\{x\})\cup\{y\}$ is an A-move in $\Pcal(\Sigma)$. Then, $|y\cap \gamma|\leq 2$. 
\end{lemma}
\begin{proof} 
Let $E$ be the 4-holed sphere where the A-move occurs. If $\gamma \cap E=\emptyset$, then $\gamma \cap y=\emptyset$. If not, then by assumption $\gamma \cap E$ is (at most) two parallel properly embedded arcs in $E$ disjoint from $x$ (see Figure \ref{fig_lem_6}). In the 4-holed sphere, the condition $|x\cap y|=2$ forces $|\gamma \cap y|=2$. 
\begin{figure}[ht!]
\labellist \large\hair 2pt 
\pinlabel {$x$}  at -5 20 
\pinlabel {{\color{red}$\gamma$}}  at 90 5 
\pinlabel {$x$}  at 113 20 
\pinlabel {{\color{red}$\gamma$}}  at 200 40 
\endlabellist \centering
\includegraphics[width=8cm]{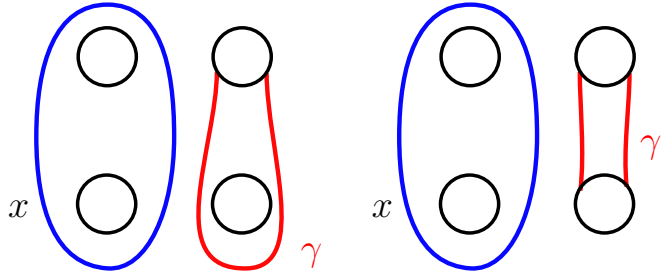}
\centering
\caption{We can think that $x$ and $\gamma$ have the same `slope' in the 4-holed sphere $E$.}
\label{fig_lem_6}
\end{figure}
\end{proof}

\begin{proposition}\label{lem_7}
Let $P_1=\gamma \cup f$ and $P_2=\psi \cup f$ be two pants decompositions for some sets of curves $\gamma, \psi, f$. Suppose $\lambda$ is a path in $\Pcal(\Sigma)$ connecting $P_1$ and $P_2$ such that each curve in $\gamma$ moves exactly once. Then, $|x\cap y|\leq 2$ for all $x\in \gamma$ and $y\in \psi$ 
\end{proposition}

\begin{proof}
We prove this by induction on the length of $\lambda$. The base case, $|\lambda|=1$, is the statement of Lemma \ref{lem_6}. Suppose that the statement is true for all paths of length at most $n\geq 1$, and let $\lambda$ be a path of length $n+1$. Label the curves in $\gamma$ and $\psi$ so that $\gamma_k \mapsto \psi_k$ is the $k$-th A-move in $\lambda$. By the strong inductive hypothesis, $|\gamma_i \cap \psi_j|\leq 2$ for all $(i,j)\neq (1,n+1)$. Let $P'=\left(P_2-\{\psi_{n+1}\}\right) \cup \{\gamma_{n+1}\}$ be the $n$-th pants decomposition in $\lambda$. As stated before, we knot that $\gamma_1$, $x=\gamma_{n+1}$, and $P'$ satisfy the hypothesis of Lemma \ref{lem_6}. Thus, by the lemma, $|\gamma_1 \cap \psi_{n+1}|\leq 2$. 
\end{proof}

\subsubsection*{Edges in $\Pcal_c(T)$}

\begin{lemma}\label{lem:gamma1_moves_first}
Let $T$ be a $b$-string tangle and let $P\in \Pcal_c(T)$ be a pants decomposition with $b-1$ even curves. There is no edge $e:x\mapsto y$ in $C^*(\Sigma)$ starting at $P$ such that $x\in P$ is even and $y$ bounds a c-disk for $T$. 
\end{lemma}

\begin{proof}
%
By Lemma \ref{lem_0}, every pants decomposition has at least $b-1$ even curves. Thus, $y$ cannot be a cut-curve and so $y$ bounds a compressing disk for $T$. 
Let $E$ be the 4-holed sphere corresponding to $e$. The four boundaries of $E$ bound disks disjoint from $E$ so we can use Part 2 of Lemma \ref{lem_0} to estimate the number of even loops in them. Since both $x$ and $y$ are even, 
out of the four boundaries of $E$, either none, two or all, bound an even number of punctures. Since $P$ has exactly $b-2$ even loops besides $x$, we conclude that none of the boundaries of $E$ bound an even number of punctures. Thus, each loop in $\partial E$ bounds a cut-disk for $T$. 
Since $P\in \Pcal_c(T)$, Lemma \ref{lem_4} implies that there exist a complete collection of shadow arcs for $T$ such that (1) all but two arcs are disjoint from $E$, (2) each of the two arcs connects a pair of boundaries of $E$, 
and (3) the shadows are disjoint from $x$. Analogously, since $(P-\{x\})\cup \{y\}\in \Pcal_c(T)$, there is a collection of shadows for $T$ satisfying the properties (1)-(3) with respect to $y$. Hence, we can treat $x$ and $y$ as compressing disks for a 2-string tangle. Since 2-string tangles admit unique compressing disks, $x=y$, contradicting the fact that $e$ is an edge in $C^*(\Sigma)$. 
\end{proof}

\section{Estimates for $\Lcal^*$-invariants}\label{section:estimates_L*}

In this section we give a lower bound for the $\Lcal^*$-invariant of $(b,c)$-bridge trisections. Notice that Lemma \ref{lem:then_nonmin} and Proposition \ref{prop_1} below hold for both $\Lcal^*$- and $\Lcal$-invariants. 

\begin{lemma}\label{lem:then_nonmin}
Let $r$ be a reducing curve for the link $T_i \cup \T_j$ and let $\left(P^i_{ik},P^k_{ik}\right)$ be a efficient defining pair for $T_i\cup \T_k$. Suppose that $|r\cap y|$ is zero for all $y\in P^i_{ik}\cap P^k_{ik}$ and at most two for all $y\in P^i_{ik}-P^k_{ik}$. Then $\mc{T}$ is reducible or stabilized. 
\end{lemma}

\begin{proof}
By Lemma \ref{lem_1}, the curves in the geodesic pair are of the form $P^i_{ik}=\psi \cup g$ and $P^k_{ik}=\psi \cup g'$. By the same lemma, there exists a geodesic $\lambda'$ of length $|g|$, from $P^i_{ik}$ to $P^k_{ik}$, that moves each curve in $g$ exactly once. By hypothesis, $\gamma$ is disjoint from all curves in $g$. 
Thus, via several applications of Lemma \ref{lem_6}, we conclude that $|r \cap y|\leq 2$ for all $y\in g'$. Hence, the hypothesis of Propositions \ref{lem_5} and \ref{lem_3} with $r$ and $P^k_{ik}$ are satisfied. 
If $r$ does not separate any pair of punctures $p\sim_k q$, then by the latter proposition, $r$ bounds a compressing disk for $T_k$. Thus, $r$ bounds a disk in all tangles and $\mc{T}$ is reducible. If $r$ separates a pair of punctures $p\sim_k q$, then by Proposition \ref{lem_5}, there is an shadow $a$ for $T_k$ with $|a\cap r|=1$. Hence, by Lemma \ref{lem_2}, $\mc{T}$ is a stabilization. 
\end{proof}

\begin{proposition}\label{prop_1}
Let $\mc{T}$ be an irreducible and unstabilized bridge trisection. Let $\gamma$ be a c-reducing curve for $T_i\cup \T_j$ and let $(P^i_{i,k},P^k_{ik})$ be an efficient defining pair for $T_i\cup \T_k$. Then $\gamma \not \in P^i_{ik}$. 
\end{proposition}

\begin{proof}
Let $\psi$, $g$ and $g'$ be sets of curves so that $P^i_{ik}=\psi \cup g$ and $P^k_{ik}=\psi\cup g'$. Lemma 2.10 of \cite{aranda2021bounding} states that $\gamma \not \in \psi$. 
Suppose 
that $\gamma\in g$. 
In particular, $r=\gamma$ satisfies the hypothesis of \ref{lem:then_nonmin}. Thus, $\mc{T}$ is reducible or stabilized.
\end{proof}

\begin{theorem}\label{thm_lower_L*}
Let $\mc{T}$ be a $(b;c)$-bridge trisection for an embedded surface $F\subset S^4$. If $\mc{T}$ is irreducible and unstabilized, then $\Lcal^*(\mc{T})\geq 3(b+c-3)$. 
\end{theorem}
\begin{proof} 
Lemma \ref{lem_1} and Proposition \ref{prop_1} imply that the distance in $\mc{C}^*(\Sigma)$ between any efficient defining pair $(P^i_{ik}, P^k_{ik})$ is at least $b-3+c$. 
\end{proof}

\begin{example}[Unknotted torus]\label{rmrk:unknotted_torus}
From \cite{meier2017bridge}, we know that the unknotted torus, denoted by $F$, admits a minimal $(3,1)$-bridge trisection. One can see that such splitting has $\Lcal$-invariant at most $3$. So, by Theorem \ref{thm_lower_L*}, we get that $3\cdot 1 \leq \Lcal^*(F)\leq \Lcal(F)\leq 3$. Thus, the $\Lcal^*$- and $\Lcal$-invariants of $F$ are equal to $3$. 
\end{example}
This example shows, in particular, that the lower bound given in Theorem \ref{thm_lower_L*} is sharp. 
We conjecture that this bound can be improved if $F$ is knotted. 
In other words, $\Lcal^*(F)\leq 3(b+c-3)$ if and only if $F$ is unknotted. In the rest of this Section we show evidence for this conjecture by computing the $\Lcal^*$-invariants for infinite families of knotted surfaces. 


\subsection*{$\Lcal^*$-invariant of $(b,2)$-trisections}

\textbf{Notation.} Let $\mc{T}$ be a $(b,2)$-bridge trisection. Using the notation of Lemma \ref{lem_1}, let $\gamma$, $\psi$, $f$, $g$ be sets of curves such that $P^i_{ij}=\gamma \cup f$, $P^i_{ik}=\psi\cup g$, and $\gamma$ (resp. $\psi$) are c-reducing curves for $T_i\cup \T_j$ (resp. $T_i\cup T_k$). The condition $c=2$ implies that $\gamma$ and $\psi$ have exactly one reducing curve each. We denote them by $\gamma_1\in \gamma$ and $\psi_1\in \psi$, respectively. A key observation is that each pants decomposition $P^i_{ij}$ and $P^i_{ik}$ has exactly $b-1$ curves bounding even punctures. 

\begin{theorem}\label{thm:L*_disco}
Let $\mc{T}$ be an irreducible and unstabilized $(b;2)$-bridge trisection. If the underlying surface is \textbf{disconnected}, then $\Lcal^*(\mc{T})\geq 3b$. 
\end{theorem}

\begin{proof}
We will show that the distance $d(P^i_{ij}, P^i_{ik})$ in $C^*(\Sigma)$ is at most $3b$. Let $\lambda$ be a path connecting $P^i_{ij}$ and $P^i_{ik}$. By Proposition \ref{prop_1}, the length of $\lambda$ is at least $|\gamma|=b-1$. Suppose, by way of contradiction, that every curve in $\gamma$ moves exactly once. In particular, $f=g$ is fixed and every curve in $\psi$ moves once. 
We now investigate when the reducing curves $\gamma_1$ and $\psi_1$ moves. By Lemma \ref{lem:gamma1_moves_first}, $\gamma_1$ cannot move to $\psi_1$. Without loss of generality, assume that $\gamma_1 \mapsto \psi_2$ and $\gamma_2\mapsto \psi_1$. 
By Lemma \ref{lem:gamma1_moves_first}, the edge $\gamma_1\mapsto \psi_2$ must occur after $\gamma_2$ moves. In particular, $\gamma_1$ and $\psi_1$ will be disjoint reducing curves. 
However, since the surface is disconnected, the curves $\gamma_1$ and $\psi_1$ bound the same punctures in $\Sigma$. This forces them to be the same curve or to intersect, contradicting the previous statement. Hence, such $\lambda$ cannot exist and $d(P^i_{ij},P^i_{ik})\geq b$. 
\end{proof}

\begin{example}[$T^2$-spin of a 2-bridge link]\label{example:T2spin_L*}
Let $L$ be a 2-bridge 2-component unlink and let $F=T(L)$. By lemmas \ref{lem:6bridge} and \ref{lem:bridget2}, every minimal bridge $(b;c_1,c_2,c_3)$-bridge trisection of $F$ is irreducible and satisfies $b=6$ and $c_i=2$. Thus, by Theorem \ref{thm:L*_disco}, $\Lcal(F)\geq 3\cdot 6=18$. We now discuss the opposite inequality.  
We use the $(6;2)$-bridge trisection for $F$ described in Lemma \ref{lem:T2_triplane}. Figure \ref{fig:t2fig1} shows a picture of the three links $L_i=T_i \cup \T_{i+1}$ when $L$ a Hopf link. 
For an arbitrary 2-bridge link, figures \ref{fig:t2fig1} and \ref{fig:t2fig2} will almost be the same except for the long loops on each pants that will then resemble the Conway number of $L$. The sequence of paths in $C^*(\Sigma)$ described in Figure \ref{fig:t2fig2} shows that $\Lcal^*(F)\leq 6+6+6=18$. Hence, $\Lcal^*(F)=18$.
\end{example}

\begin{figure}[ht!]
\includegraphics[width=17cm]{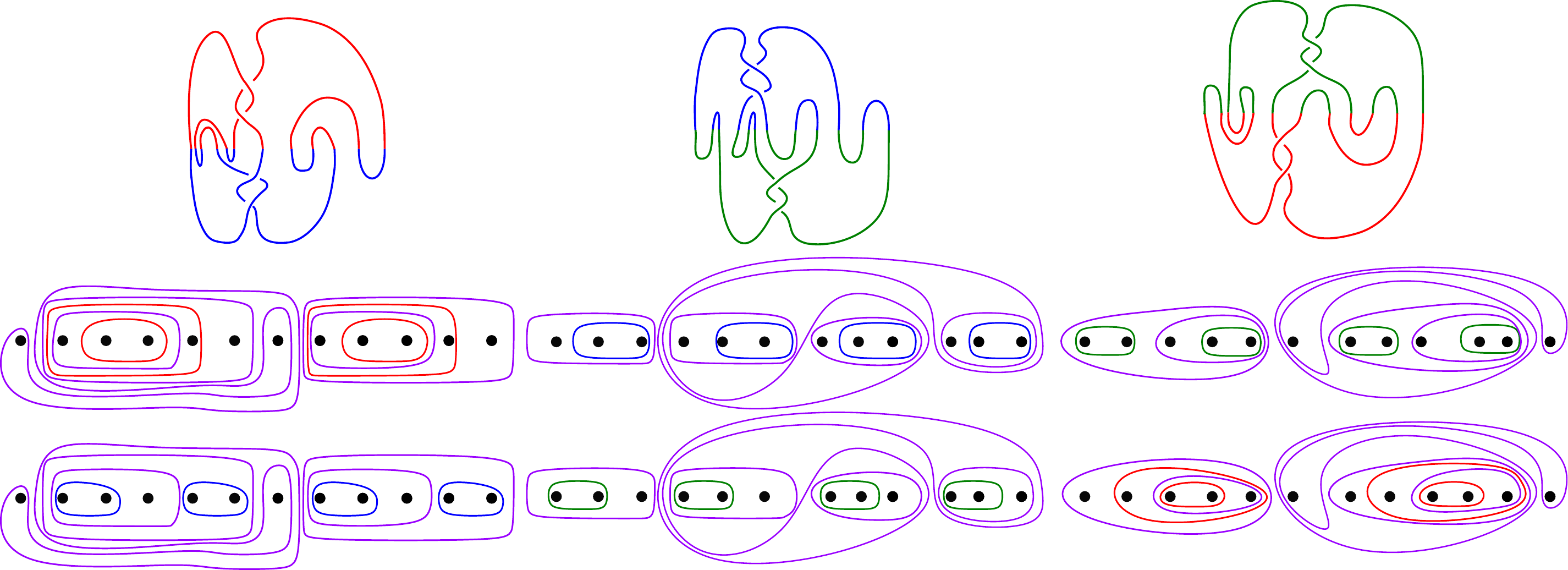}
\centering
\caption{$\Lcal^*$-invariant of the $T^2$-spin of a 2-bridge link. Each column depicts bridge positions and efficient defining pairs for the unlink $L_i = T_{i}\cup \T_{i+1}$.}
\label{fig:t2fig1}
\end{figure}
\begin{figure}[ht!]
\includegraphics[width=15cm]{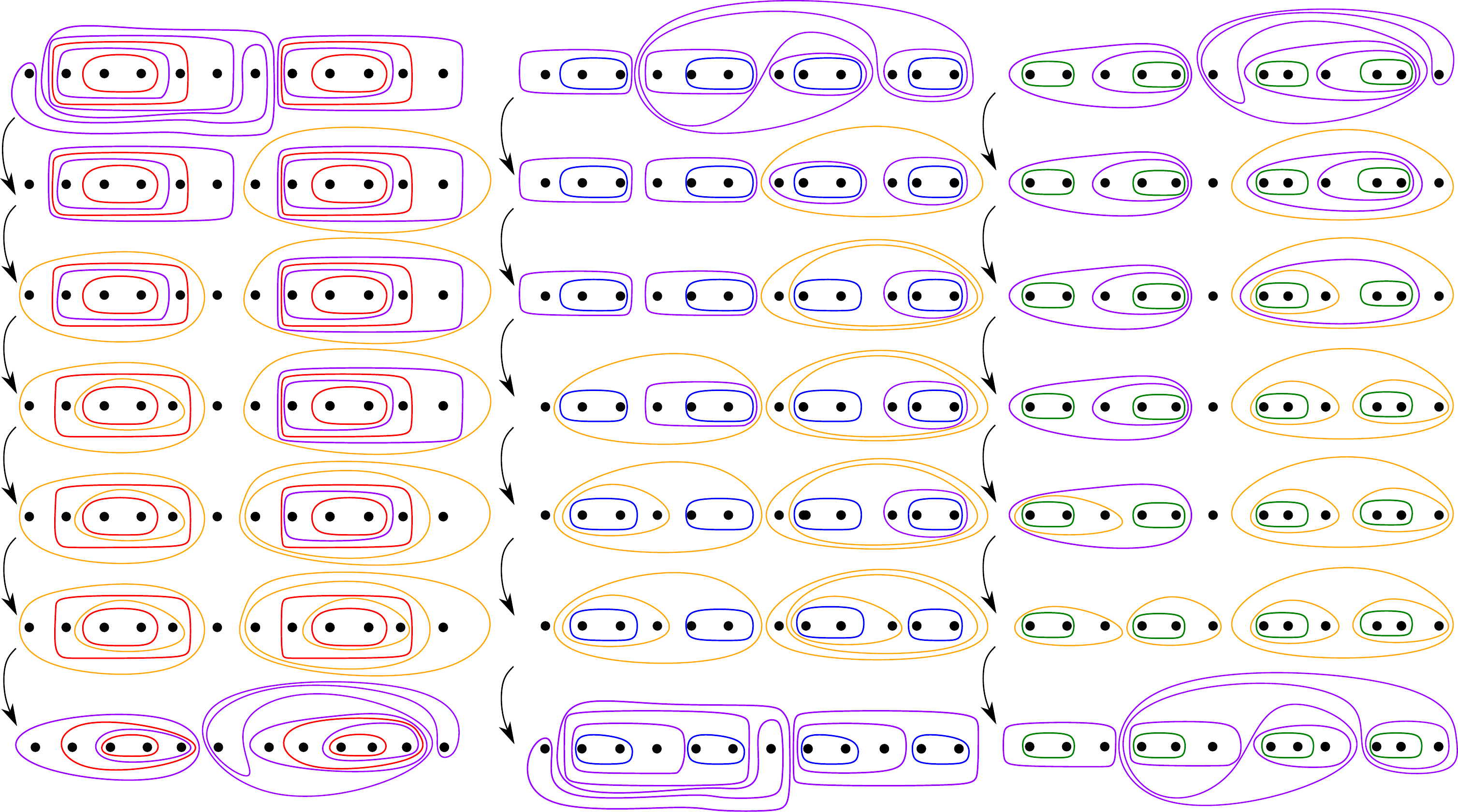}
\centering
\caption{$\Lcal^*$-invariant of the $T^2$-spin of a 2-bridge link. A path 
connecting $P^i_{ij}$ and $P^i_{ik}$. }
\label{fig:t2fig2}
\end{figure}

\begin{lemma}\label{lem:wheregamma2goes}
Let $F$ be a non-split surface with minimal $(4,2)$-bridge trisection. Suppose $e$ is an edge with initial endpoint in $P_{ij}^i$. 
Then, $e$ does not move loops in $\gamma$ to any reducing curves in $\psi$. 
\end{lemma}
\begin{proof}
By Lemma \ref{lem:gamma1_moves_first}, the case when a reducing curve in $\gamma$ moves cannot occur. 
Suppose then that $\gamma_2$ moves to $\psi_1$. In particular, $\gamma_1$ and $\psi_1$ are disjoint. 
Since $F$ is connected, Lemma 3.6 of \cite{aranda2021bounding} states that if the reducing curves $\gamma_1$ and $\psi_1$ bound distinct number of punctures, then they intersect at least four times. Thus, $\gamma_1$ cuts off a twice punctured disk $D$ distinct from the twice punctured disk cut off by $\psi_1$. The disk $D$ is a regular neighborhood of a shadow $s$ for $T_i$. Since $\gamma_1$ is a reducing curve for $T_i\cup \T_j$, $s$ is also a shadow for an arc in $T_j$. 

On the other hand, $\psi_1$ is a reducing curve for $T_i \cup \T_k$ and $s$ is a shadow for the 3-bridge unknot component $U$ of $T_i \cup \T_k$. Since $U$ is an unknot, there exists a complete collections of shadows, denoted by $a_i$ and $a_k$, for the arcs in $T_i\cap U$ and $T_k\cap U$ satisfying: $a_i\cup a_k$ is an embedded circle. 
Theorem 1.1(2b) of \cite{hayashi1998heegaard} states that such shadows can be picked so that $s\in a_i$. In particular, there is a shadow $b\in a_k$ for $T_k$ satisfying $s\cap b=\partial s \cap \partial b=\{pt\}$. According to Lemma 6.2 of \cite{meier2017bridge}, the bridge trisection is stabilized, contrary to our assumption.
%
\end{proof}

\begin{theorem}\label{thm:L*_disco_42}
Let $\mc{T}$ be an irreducible and unstabilized $(4;2)$-bridge trisection. Then $\Lcal^*(\mc{T})\geq 12$. 
\end{theorem}
\begin{proof}
The result will follow once we show that $d(P_{ij}^i,P^i_{ik}) \geq 4$. 
We use the notation at the beginning of the sub-subsection. 
By Proposition \ref{prop_1}, we know that each curve in $\gamma$ moves at least once, giving $d(p_{ij}^i,p^j_{ij}) \geq 3.$
By way of contradiction, suppose that $\lambda$ is a path of length 3 in $C^*(\Sigma)$ that connects $P_{ij}^i$ and $P_{ik}^j$. 
In particular, $f=g$ 
and, by Lemma \ref{lem:gamma1_moves_first}, $\gamma_1$ cannot move first. Suppose that $\gamma_2$ moves first. Note that $\gamma_2$ cannot move to $\psi_1$ because of Lemma \ref{lem:wheregamma2goes}. Hence, $\gamma_2$ moves to a cut-curve in $\psi$, say $\gamma_2\mapsto\psi_2$.

As we did in the proof of Lemma \ref{lem:gamma1_moves_first}, the fact that 2-string tangles have a unique compressing disk implies that $\gamma_1$ cannot be moved to $\psi_1$ along $\lambda$. Thus $\gamma_1 \mapsto \psi_3$ and $\gamma_3 \mapsto \psi_1$. By the lack of even curves (as in Lemma \ref{lem:gamma1_moves_first}), we conclude that $\gamma_1$ cannot move before $\gamma_3$. Hence, $\gamma_1\mapsto \psi_3$ is the last move of $\lambda$, contradicting Lemma \ref{lem:wheregamma2goes}.
%
\end{proof}

\subsubsection*{Examples of $(4,2)$-bridge trisections}

\begin{example}[Spin of a 2-bridge link]\label{exam:L*_2b}
Let $F$ be the spin of a 2-bridge link (or knot). By Proposition \ref{prop:spun2b}, $b(F)=4$ and every minimal $(4;c_1,c_2,c_3)$-bridge trisection is irreducible and satisfies $c_i=2$. Thus, by Theorem \ref{thm:L*_disco_42}, we get that $\Lcal^*(F)\geq 12$. We now show the opposite inequality. 
We use the $(4;2)$-bridge trisection $\mc{T}_{MZ}$ for $F$ given in Section 5.1 of \cite{meier2017bridge}. Figure \ref{fig:dualtrefoilpath} shows a picture of the three links $L_i=T_i \cup \T_{i+1}$ for $\mc{T}_{MZ}$ when $L$ a Hopf link. 
For an arbitrary 2-bridge link, figures \ref{fig:dualtrefoilpath} and \ref{fig_spun_hopf_1} will almost be the same except for the long loops on each pants that will then resemble the Conway number of $L$. The sequence of paths in $C^*(\Sigma)$ described in Figure \ref{fig_spun_hopf_1} shows that $\Lcal^*(F)\leq 12$. 
Another way to see the upper bound is as follows. The proof of Corollary 4.5 of \cite{aranda2021bounding} built paths in $\Pcal(\Sigma)$ to conclude that the $\Lcal$-invariant is at most $6d(p/q,0)+6$ where $p/q$ is the Conway number of $L$. Since $\Pcal(\Sigma)$ is a subgraph of $C^*(\Sigma)$, the upper bound is still valid for $\Lcal^*(F)$. Furthermore, any two essential simple closed curves in a 4-holed sphere are connected by an edge in $C^*(\Sigma_{0,4})$. Thus, $\Lcal^*(F)\geq 6\cdot 1 + 6 = 12$. Hence, $\Lcal^*(F)=12$. 
\end{example}

\begin{example}[2-twist spun trefoil]\label{example:yoshi_L*}
Let $F$ be the $k$-twist spun of a $(2,2k-1)$-torus knot. By Theorem 1.2 of \cite{joseph2021bridge}, the meridional rank of $F$ is 2. Moreover, Meier and Zupan were able to produce $(4,2)$-bridge trisection diagrams of $F$ (see Figure 22 of \cite{meier2017bridge}), which implies that $b(F)=4$. Due to an Euler characteristic argument, we have that every minimal $(4;c_1,c_2,c_3)$-bridge trisection is irreducible and satisfies $c_i=2$. Thus, by Theorem \ref{thm:L*_disco_42}, we get that $\Lcal^*(F)\geq 12$. When $k=2$ (that is, when $F$ is the 2-twist spun trefoil), we are able to show the opposite inequality by consider the diagram $10_2$ of the in Yoshikawa's table \cite{yoshikawa1994enumeration}. 

Recall that to get a bridge trisection diagram, one can start by putting a banded unlink in the banded bridge splitting. Figure \ref{fig:Yoshi1} (left) is Yoshikawa's diagram. The black markers represent bands so that we get a banded link diagram in \ref{fig:Yoshi1} (middle), which admits a banded bridge splitting. The corresponding tri-plane diagram is given in Figure \ref{fig:Yoshi1} (right). 
Figures \ref{fig:Yoshi2} and \ref{fig:Yoshi3} show that $\Lcal^*(F)\leq 12$. Hence, the $\Lcal^*$-invariant of a 2-twist spun trefoil is equal to $12$.
\end{example}

\begin{question}
Do all non-trivial $k$-twist spun 2-bridge knots have $\mathcal{L}^*$-invariant equal to $12$?
\end{question}
\begin{figure}[ht!]
\includegraphics[width=12cm]{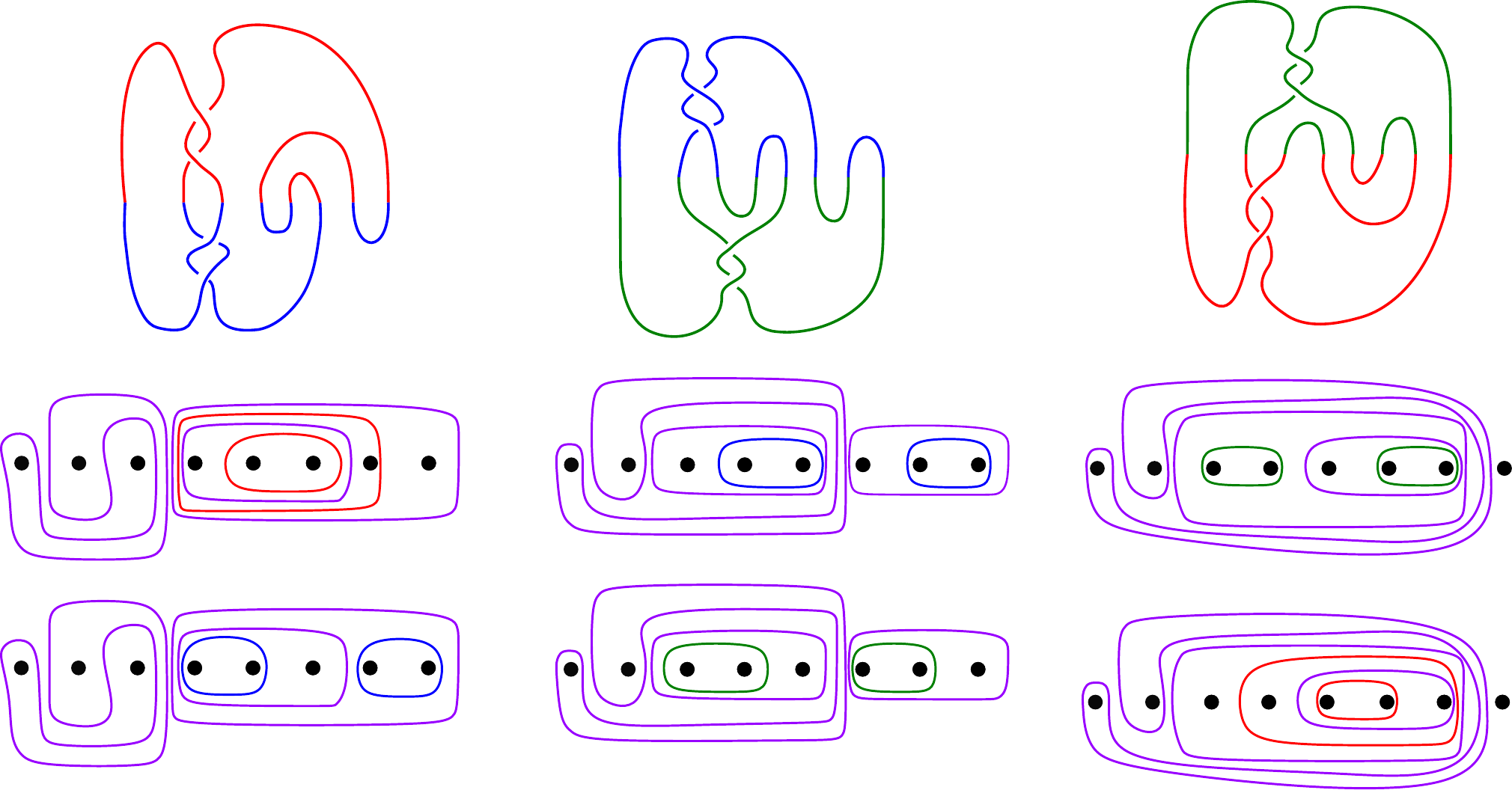}
\centering
\caption{$\Lcal^*$-invariant of the spin of the Hopf link. Each column depicts bridge positions and efficient defining pairs for the unlink $L_i = T_{i}\cup \T_{i+1}$. 
}
\label{fig_spun_hopf_1}
\end{figure}
\begin{figure}[ht!]
\includegraphics[width=15cm]{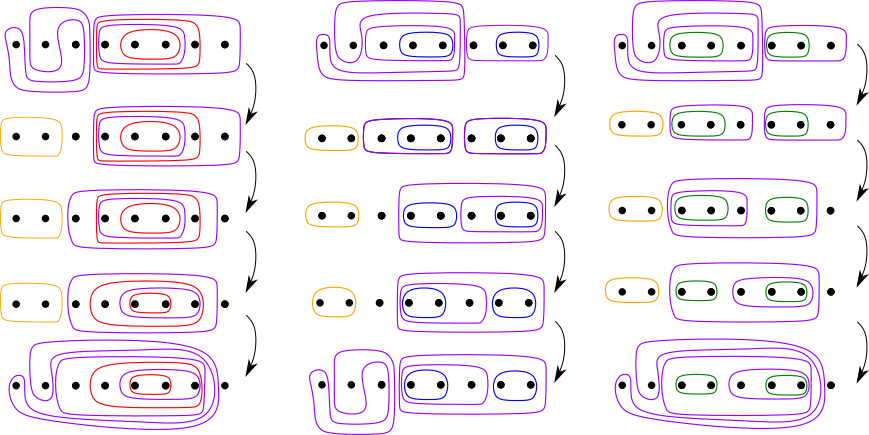}
\centering
\caption{$\Lcal^*$-invariant of the spin of the Hopf link. A path of length four connecting $P^i_{ij}$ and $P^i_{ik}$. 
}
\label{fig:dualtrefoilpath}
\end{figure}
\begin{figure}[ht!]
\includegraphics[width=12cm]{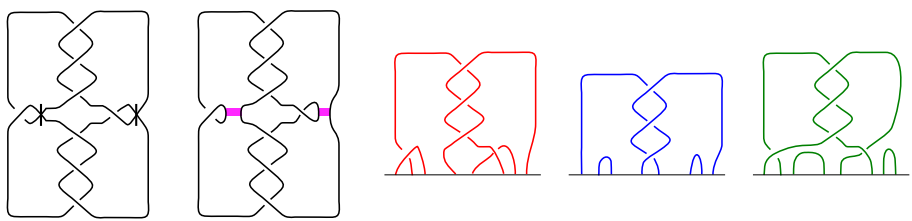}
\centering
\caption{(Left) Yoshikawa's diagram of $10_2$, which represents the 2-twist spun trefoil. (Middle) A banded bridge splitting of $10_2$. (Right) A tri-plane diagram of $10_2$.}
\label{fig:Yoshi1}
\end{figure}

\begin{figure}[ht!]
\includegraphics[width=12cm]{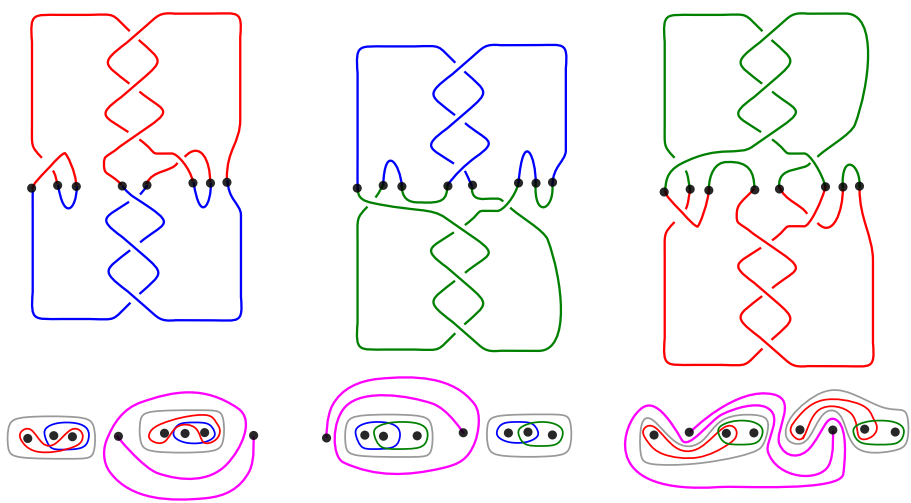}
\centering
\caption{$\Lcal^*$-invariant of the 2-twist spun trefoil. 
Efficient defining pairs for the unlink $T_{i}\cup \T_{i+1}$.}
\label{fig:Yoshi2}
\end{figure}

\begin{figure}[ht!]
\includegraphics[width=12cm]{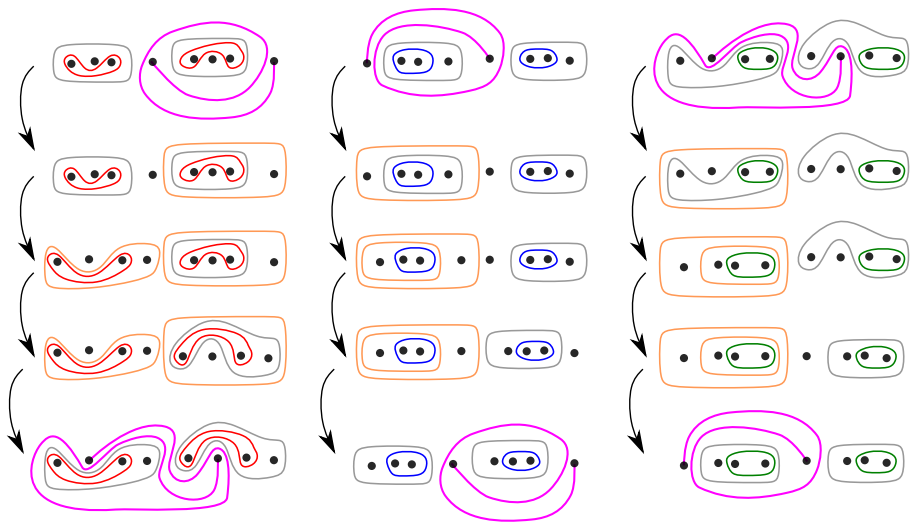}
\centering
\caption{$\Lcal^*$-invariant of the 2-twist spun trefoil. A path 
connecting $P^i_{ij}$ and $P^i_{ik}$.}
\label{fig:Yoshi3}
\end{figure}

\section{Estimates for $\Lcal$-invariants. }\label{section:estimates_L}

In this section we give a lower bound for the $\Lcal$-invariant of $(b,c)$-bridge trisections. The careful reader might notice that these techniques can be used to prove the estimates for general $(b;c_1,c_2,c_3)$-bridge trisections. 

\begin{proposition}\label{prop_at_least_twice}
Let $T$ be irreducible and unstabilized $(b;c)$-bridge trisection. Let $(P^i_{ij}, P^j_{ij})$ and $(P^i_{ik},P^k_{ik})$ be efficient defining pairs for $T_i\cup \T_j$ and $T_i\cup \T_k$, respectively. If $c\geq 2$, then the distance $d(P^i_{ij}, P^i_{ik})$ in $\Pcal(\Sigma)$ is at least $b+c-2$. 
\end{proposition}

\begin{proof}
Let $\gamma$, $\psi$, $f$, $g$, and $g'$ be sets of curves in $\Sigma$ such that $P^i_{ij}=\gamma \cup f$, $P^i_{ik}=\psi\cup g$, and $P^k_{ik}=\psi\cup g'$ as in Lemma \ref{lem_1}. By Proposition \ref{prop_1}, $\gamma\cap P^i_{ik}=\emptyset$, so $d_{\Pcal(\Sigma)}(P^i_{ij}, P^i_{ik})\geq |\gamma|=b+c-3$. 
Suppose, by way of contradiction, that $\lambda$ is a path of length $|\gamma|$ in $\Pcal(\Sigma)$ connecting $P^i_{ij}$ with $P^i_{ik}$. Such path must fix $f$ and move each curve in $\gamma$ exactly once. By Proposition \ref{prop_1}, $\psi \cap P^i_{ij}=\emptyset$, so $f=g$ and $|\gamma|=|\psi|$. 
Let $\gamma_1\in \gamma$ be a reducing curve for $T_i\cup \T_j$, this exists since $c\geq 2$. By Proposition \ref{lem_7} applied to the path $\lambda$, $|\gamma_1\cap x|\leq 2$ for all $x\in P^i_{ik}$. 
%
By Lemma \ref{lem:then_nonmin}, $\mc{T}$ is reducible or an stabilization. 
\end{proof}

\begin{theorem}\label{thm_lower_L}
Let $\mc{T}$ be a $(b;c)$-bridge trisection with $c\geq 2$. If $\mc{T}$ is irreducible and unstabilized, then $\Lcal(\mc{T})\geq 3(b+c-2)$. 
\end{theorem}
\begin{proof} 
This follows from Proposition \ref{prop_at_least_twice}. 
\end{proof}

\begin{proposition}\label{prop_improved_lower_bound_c2}
Let $T$ be irreducible and unstabilized $(b;2)$-bridge trisection. Let $(P^i_{ij}, P^j_{ij})$ and $(P^i_{ik},P^k_{ik})$ be efficient defining pairs for $T_i\cup \T_j$ and $T_i\cup \T_k$, respectively. Then the distance $d(P^i_{ij}, P^i_{ik})$ in $\Pcal(\Sigma)$ is at least $b+1$. 
\end{proposition}

\begin{proof} 
We use the same notation as in Proposition \ref{prop_at_least_twice}. The condition $c=2$ implies that there is at most one reducing curve in $\gamma$ and $\psi$. 
By Lemma \ref{lem_1}, $|f|=b-2$ and $|\gamma|=b-1$. Thus, by Lemma \ref{lem_0}, there are exactly $b-2$ cut curves in $\gamma$ (resp. $\psi$) and one reducing curve in $\gamma$ (resp. $\psi$). Let $\gamma_1\in \gamma$ and $\psi_1\in \psi$ be the reducing curves. 
Proposition \ref{prop_at_least_twice} shows that the distance between $P^i_{ij}$ and $P^i_{ik}$ is at most $b$. Suppose, by way of contradiction, that $\lambda$ is a path in $\Pcal(\Sigma)$ of length $b=|\gamma|+1$. By Proposition \ref{prop_1}, each curve in $\gamma$ must move at least once. We have two cases depending on whether or not only curves in $\gamma$ move. 

\textbf{Case 1.} A curve in $f$ moves. Let $f_1\in f$ and $g_1\in g$ be the two extra curves that move along $\lambda$. Every curve in $\gamma \cup \{f_1\}$ moves exactly once to $\psi\cup \{g_1\}$, and the set $f-\{f_1\}=g-\{g_1\}$ is fixed. 
Suppose that the reducing curve $\gamma_1$ is disjoint from $g_1$. In particular, $\gamma_1$ is disjoint from all curves in $g$. Proposition \ref{lem_7} applied to the path $\lambda$ implies that $|\gamma_1\cap x|\leq 2$ for all $x\in P^i_{ik}$. Here, 
Lemma \ref{lem:then_nonmin} implies that $\mc{T}$ is a reducible or an stabilization. 
Hence, $\gamma_1$ intersects $g_1$. Analogously, $\psi_1$ must intersect $f_1$. 
This implies that the A-move involving $\gamma_1$ must occur before $g_1$ appears, and the A-move involving $\psi_1$ must occur after $f_1$ moves. 

We focus on the A-move containing $f_1$. Let $E$ be the 4-holed sphere where the A-move $f_1 \mapsto y$ occurs and let $P$ be the pants decomposition containing $f_1$. 
There are at most $b+1$ curves in $P$ bounding an even number of punctures: $\gamma_1$, $\psi_1$, $f_1$, $g_1$ and $f_i=g_i$ for $2\leq i\leq b-2$. Since, $g_1\cap \gamma_1\neq \emptyset$ and $f_1 \cap \psi_1 \neq \emptyset$, we know that $\psi_1\not \in P$ and at most one of $\{g_1, \gamma_1\}$ belongs to $P$. 
By Lemma \ref{lem_0}, every pants decomposition has at least $b-1$ curves bounding even punctures. Thus, one of $\{g_1, \gamma_1\}$, say $x$, belongs to $P$ and $P$ has exactly $b-1$ even curves. This implies that $y$ is not a cut curve and $y\in \{g_1, \psi_1\}$.  
The edge $f_1 \mapsto y$ satisfies the hypothesis of Lemma \ref{lem:gamma1_moves_first}. By the lemma, such move cannot occur. 

\textbf{Case 2.} The set $f=g$ is fixed. Let $\theta$ be the pivot curve. In other words, for some $x\in \gamma$, $y\in \psi$, each curve in $\gamma-\{x\}$ moves once to the set $\psi-\{y\}$, $x$ moves to $\theta$, and $\theta$ moves to $y$. 
Let $P$ be the pants decomposition in $\lambda$ 
that contains $\theta$. Let $\lambda_j$ be the sub-path of $\lambda$ connecting $P^i_{ij}$ and $P$, and let $\lambda_k$ be the path connecting $P$ and $P^i_{ik}$. Observe that $\lambda_i$ and $\lambda_j$ satisfy the hypothesis of Lemma \ref{lem_7}. In particular, if $\gamma_1\in P$, then $|\gamma_1 \cap z|\leq 2$ for all $z\in \psi$. Notice that, since $g$ is fixed, $\gamma_1$ is disjoint from each curve in $g$. 
Then, by Lemma \ref{lem:then_nonmin}, $\mc{T}$ is reducible or stabilized. 
Hence, $\gamma_1\not\in P$ and, by a similar argument, $\psi_1 \not \in P$. 
On the other hand, there are $b-1$ possible curves in $P$ bounding an even number of punctures: $\{\theta\}\cup f$. Since there are at least $b-1$ even curves in $P$ (Lemma \ref{lem_0}), $\theta$ must bound an even number of punctures. By the same reasoning, $\theta$ cannot move to a cut curve and so $\theta$ must move to $\psi_1$. We can also conclude that $\gamma_1$ moves to $\theta$. 

By the previous paragraph, $|\gamma_1 \cap y|\leq 2$ for all $y\in \psi$ that appear before $\psi_1$. Suppose $|\gamma_1\cap \psi_1|\leq 2$ and let $x\mapsto y$ be the edge after $\theta\mapsto \psi_1$. Since $x\in \gamma$, $x$ and $\gamma_1$ are disjoint. 
Hence, every curve $z$ in the pants decomposition containing $x$ satisfy $|\gamma_1 \cap z|\leq 2$. Thus, by Lemma \ref{lem_6}, $|\gamma_1 \cap y|\leq 2$. By induction, one can show that $|\gamma_1 \cap y|\leq 2$ for all $y\in \psi$. Now, since $g=f$, $\gamma_1$ is disjoint from $g$. Hence, 
by Lemma \ref{lem:then_nonmin}, $\mc{T}$ is reducible or stabilized.
Therefore, 
\textbf{if we can show that $|\gamma_1 \cap \psi_1|\leq 2$, the proof of this Proposition will be over. }

Let $E$ be the 4-holed sphere where the A-move $\gamma_1 \mapsto \theta$ occurs and let $Q$ and $P$ be the pants decompositions containing $\gamma_1$ and $\theta$, respectively. Since both curves are even and $P$ has exactly $b-1$ even curves, all four boundaries of $E$ bound an odd number of punctures; i.e., they are cut curves for $T_i$ or boundaries of small disks around one puncture. 
We then apply Proposition \ref{lem_5} with the pants $Q$ and the loop $\theta$. 
There exist punctures $\{p_1,p_2,p_3,p_4\}$ and shadows $\{a,b\}$ for arcs in $T_i$ satisfying $|a\cap \theta|=1$ and $|b\cap \theta|=1$ (see Figure \ref{fig:shadows}(a)). 
Since the curves in $\partial E=\{\partial_l\}_{l=1}^4$ are odd, by the same Proposition, any two other punctures connected by $T_i$ have shadows disjoint from $E$. 

\textbf{Claim 1.} Let $R$ be the pants decomposition containing $\theta$ corresponding to a vertex of $\lambda$ between $\gamma_1\mapsto \theta$ and $\theta \mapsto \psi_1$. There are four curves in $R$, denoted by $\{\partial_l\}_{l=1}^4$ such that (0) $\{\theta, \partial_1, \partial_4\}$ and $\{\theta, \partial_2, \partial_3\}$ are the boundaries of pairs of pants in $\Sigma -R$, (1) $\partial_l$ separates $p_l$ from $\{p_s: s\neq l\}$, (2) $|a\cap \partial_l|$ is one for $l=1,2$ and zero for $l=3,4$, (3) $|b\cap \partial_l|$ is one for $l=3,4$ and zero for $l=1,2$, (4) $|\gamma_1 \cap \partial_l|\leq 2$ for all $l=1,2,3,4$, and (5) $\gamma_1$ intersects at most one curve from the sets $\{\partial_1, \partial_3\}$ and $\{\partial_2, \partial_4\}$. 

If $R=P$, this is the setup discussed in the previous paragraph (Figure \ref{fig:shadows}(a)). 
Suppose, by induction, that the Claim is true for the first $t\geq 0$ pants decompositions after $P$. 
Let $x\mapsto y$ be the $t$-th edge after $P$ with $x\in R$ and suppose $x\in \gamma-\{\gamma_1\}$, $y\in \psi-\{\psi_1\}$. By the inductive hypothesis, $R$ satisfies (0)-(4) with respect to some curves $\{\partial_l\}_{l=1}^4$. We want to show the claim for $R'=(R-\{x\})\cup \{y\}$. Let $E'$ be the 4-holed sphere corresponding to this A-move, denote the boundaries of $E'$ by $\{\delta_l\}_{l=1}^4$. 

Suppose first that $\theta$ is a component of $\partial E'$. Without loss of generality, assume that $E'$ is on the same side of $\theta$ than $\{p_1, p_4\}$. 
Since $x\in \gamma$ is a cut-curve, it must separate $p_1$ and $p_4$. Without loss of generality, assume that $p_4$ and $\theta$ are on the same side of $x$ in $E'$. After a surface diffeomorphism, we can depict the curves in $E$ as in Figure \ref{fig:shadows}(b).  Moreover, since $\{\partial_l\}_{l=1}^4$ and $\partial E'$ are curves in $R$, $x$ has to be equal to $\partial_1$. 
In particular, $x=\partial_1$ is disjoint from $b$ and intersects $a$ once. Since $x$ bounds an odd number of punctures, $\delta_4$ has to be even, so $\delta_4 \in f=g$ and the 1-manifolds $\theta$, $\gamma_1$, $a$ and $b$ are disjoint from $\partial_4$.  
On the other hand, we know that $|\gamma_1\cap \theta|=2$, $\gamma_1 \cap x=\emptyset$, $|\gamma_1 \cap \partial 4|\leq 2$, and $\gamma_1$ separates $\{\partial_1, \partial_4\}$. Thus, after a surface diffeomorphism, we can draw the curves in $E'$ as in Figures \ref{fig:shadows}(b) or (c). To end, since $y\in \psi$ is a cut-curve for $T_i$, $y$ separates $\{p_1,p_4\}$. Hence, since $|\gamma_1 \cap y|\leq 2$, the condition $|x\cap y|=2$ forces $y$ to look as in Figures \ref{fig:shadows}(b)-(c). 
Let $\partial'_1=\delta_2$, $\partial'_2=\partial_2$, $\partial'3_3=\partial_3$, and $\partial'_4=y$. One can see that $\{\partial'_l\}_{l=1}^4$ satisfy conditions (0)-(5) with respect to the pants decomposition $R'=(R-\{x\})\cup \{y\}$, as desired. 

Suppose now that $\theta$ is not a component of $\partial E'$. Conditions (0) and (1) imply that $x\not\in \{\partial_l\}_{l=1}^4$ and that at most one curve in $\{\partial_l\}_{l=1}^4$ is a component of $\partial E'$. In particular, the same set of curves $\{\partial_l\}_{l=1}^4$ works for $R'$. This concludes the proof of Claim 1. 
\begin{figure}[ht!]
\labellist \small\hair 2pt  
\pinlabel {(a)} at 01 95
\pinlabel {$a$} at 50 75
\pinlabel {$b$} at 51 26
\pinlabel {{\color{red}$\gamma_1$}} at 85 90
\pinlabel {$\theta$} at 42 8

\pinlabel {(b)} at 115 95
\pinlabel {$\delta_4$} at 182 20
\pinlabel {$a$} at 155 75
\pinlabel {$b$} at 140 40
\pinlabel {$y$} at 210 65
\pinlabel {$x$} at 205 45
\pinlabel {{\color{red}$\gamma_1$}} at 115 5

\pinlabel {(c)} at 225 95
\pinlabel {$\delta_4$} at 290 20
\pinlabel {$a$} at 263 75
\pinlabel {$b$} at 248 40
\pinlabel {$y$} at 318 65
\pinlabel {$x$} at 313 45
\pinlabel {{\color{red}$\gamma_1$}} at 230 40

\scriptsize 

\pinlabel {$p_1$} at 24 75
\pinlabel {$p_2$} at 70 75
\pinlabel {$p_3$} at 24 25
\pinlabel {$p_4$} at 70 25
\pinlabel {$\partial_1$} at 35 57
\pinlabel {$\partial_2$} at 85 62
\pinlabel {$\partial_3$} at 15 34
\pinlabel {$\partial_4$} at 83 32

\pinlabel {$\delta_1=\theta$} at 137 70
\pinlabel {$\delta_2=\partial_1$} at 183 70
\pinlabel {$\delta_3=\partial_4$} at 137 20

\pinlabel {$\delta_1=\theta$} at 245 70
\pinlabel {$\delta_2=\partial_1$} at 291 70
\pinlabel {$\delta_3=\partial_4$} at 245 20

\endlabellist \centering
\includegraphics[width=15cm]{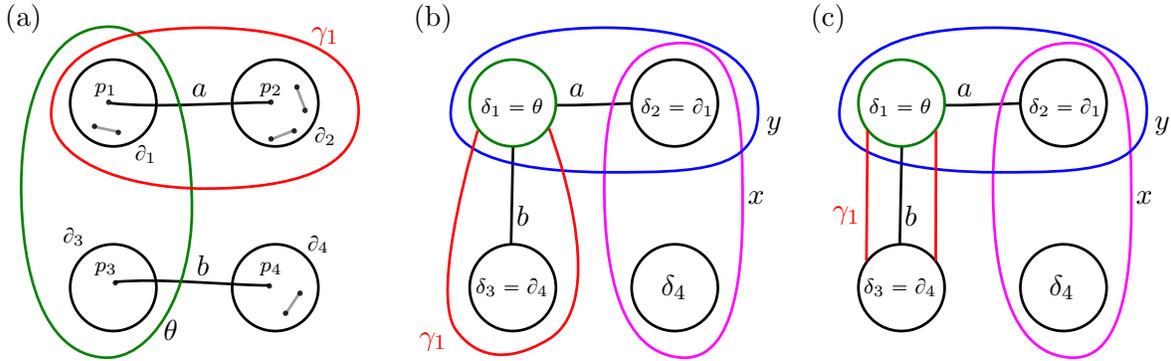}
\centering
\caption{Close look at the curves in $E$ and $E'$.}
\label{fig:shadows}
\end{figure}

Let $R$ be the pants decomposition containing $\theta$ that corresponds to the A-move $\theta\mapsto \psi_1$. By Claim 1, there are curves $\{\partial_l\}_{l=1}^4$ in $R$ satisfying (0)-(5). Moreover, the boundaries of the 4-holed $E''$ corresponding to this A-move are $\{\partial_l\}_{l=1}^4$. 
Recall that $\psi_1$ bounds a compressing disk for $T_i$ and that $p_1\sim_i p_2$. Thus, $\psi_1$ is a curve in $E''$ bounding $\{\partial_1, \partial_2\}$. The condition $|\theta\cap \psi_1|=2$ forces $\psi_1$ to have slope $n/1$ for some $n\in \mathbb{Z}$ (See Figure \ref{fig:shadows_2}(a)). 
By Proposition \ref{lem_5} applied to the pants decomposition $(R-\{\theta\})\cup \{\psi_1\} \in \Pcal_c(T_i)$ and the curve $\theta$, there are shadows $\{a',b'\}$ for arcs in $T_i$ such that conditions (1)-(3) of Claim 1 are satisfied, (4') $\psi_1 \cap (a'\cup b')=\emptyset$, and (5') $|a'\cap \theta|=|b'\cap \theta|=1$ (See Figure \ref{fig:shadows_2}(a)). 
By the same proposition, all other shadows for $T_i$ can be chosen to be disjoint from $E''$. 
Notice that $\{a,b\}$ and $\{a',b'\}$ are shadows for the same arcs in $T_i$. Thus, together, must form a collection of shadows for a 2-component unlink inside $T_i \cup \T_i$. This is impossible 
unless $a\cup b$ and $a'\cup b'$ can be made isotopic inside $E''$. In other words, $\psi_1$ must be disjoint from $a\cup b$, and so it has slope $0/1$ in $E''$ (See Figure \ref{fig:shadows_2}(b)). To end, Recall that the curve $\gamma_1$ satisfies conditions (4), (5), $\gamma_1 \cap (a\cup b)=\emptyset$, and $|\gamma_1 \cap \theta|=2$. There are finitely many ways to draw such $\gamma_1$ curve in Figure $E''$ (See Figure \ref{fig:shadows_2}(b)-(c)). In all such, notice that $|\gamma_1\cap \psi_1|$ is at most two. 
We then reach a contradiction using the argument in paragraph 2 of this case. This finishes the proof of Proposition \ref{prop_improved_lower_bound_c2}.

\begin{figure}[ht!]
\labellist \small\hair 2pt  
\pinlabel {(a)} at -3 105
\pinlabel {$a$} at 60 88
\pinlabel {$a'$} at 100 65
\pinlabel {$b$} at 60 39
\pinlabel {$b'$} at 70 59
\pinlabel {$\theta$} at 32 8
\pinlabel {$\psi_1$} at 57 3

\pinlabel {(b)} at 122 105
\pinlabel {$a$} at 181 87
\pinlabel {$b$} at 180 38
\pinlabel {$\psi_1$} at 220 65
\pinlabel {$\theta$} at 170 18
\pinlabel {{\color{red} $\gamma_1$}} at 188 64

\pinlabel {(c)} at 235 105
\pinlabel {$a$} at 291 87
\pinlabel {$b$} at 290 38
\pinlabel {$\psi_1$} at 335 68
\pinlabel {$\theta$} at 280 18
\pinlabel {{\color{red} $\gamma_1$}} at 335 48

\scriptsize 

\pinlabel {$p_1$} at 27 89
\pinlabel {$p_2$} at 78 89
\pinlabel {$p_3$} at 32 32
\pinlabel {$p_4$} at 86 32 
\pinlabel {$\partial_1$} at 24 78
\pinlabel {$\partial_2$} at 80 78
\pinlabel {$\partial_3$} at 22 28
\pinlabel {$\partial_4$} at 78 28

\pinlabel {$\partial_1$} at 146 80
\pinlabel {$\partial_3$} at 152 25
\pinlabel {$\partial_2$} at 212 75
\pinlabel {$\partial_4$} at 202 28
\pinlabel {$p_1$} at 158 88
\pinlabel {$p_3$} at 150 38
\pinlabel {$p_2$} at 195 88
\pinlabel {$p_4$} at 195 38

\pinlabel {$\partial_1$} at 258 80
\pinlabel {$\partial_3$} at 264 25
\pinlabel {$\partial_2$} at 326 77
\pinlabel {$\partial_4$} at 314 28
\pinlabel {$p_1$} at 270 88
\pinlabel {$p_3$} at 262 38
\pinlabel {$p_2$} at 307 88
\pinlabel {$p_4$} at 307 38

\endlabellist \centering
\includegraphics[width=15cm]{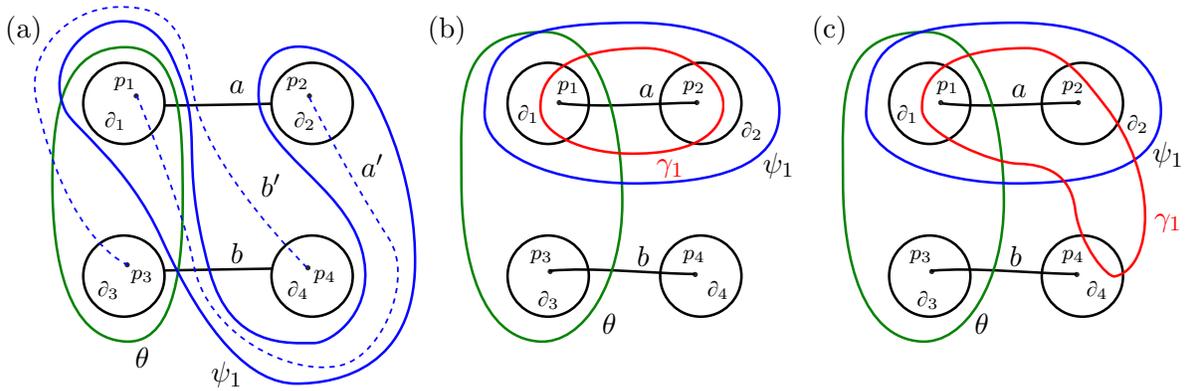}
\centering
\caption{Close look at the curves in $E''$.}
\label{fig:shadows_2}
\end{figure}
\end{proof}

$\quad$\\$\quad$

\begin{theorem}\label{thm_lower_L_c2}
Let $\mc{T}$ be a $(b;2)$-bridge trisection for an embedded surface $F\subset S^4$. If $\mc{T}$ is irreducible and unstabilized, then $\Lcal(\mc{T})\geq 3(b+1)$. 
\end{theorem}
\begin{proof} 
This follows from Proposition \ref{prop_improved_lower_bound_c2}. 
\end{proof}

\subsubsection*{Computations of $\Lcal$-invariants}

\begin{example}[Spin of (2,n)-torus links]\label{example:spin_toruslink_L}
Let $n\in \mathbb Z$ with $|n|\geq 2$. Let $L$ be a {$(2,n)$-torus} link (or knot) and let $F=S(L)$. 
In particular $L$ is a 2-bridge knot with Conway number 1/n. 
Proposition \ref{prop:spun2b} implies that every minimal $(b;c_1,c_2,c_3)$-bridge trisection for $F$ is irreducible and satisfies $b=4$ and $c_i=2$. Thus, By Theorem \ref{thm_lower_L_c2}, $\Lcal(F)\geq 3(4+1)=15$. On the other hand, Corollary 4.5 of \cite{aranda2021bounding} states that $\Lcal(F)\leq 6d(1/n,0)+9=6\cdot 1 +9=15$. Hence $\Lcal(F)=15$. 
\end{example}

\begin{example}[$T^2$-spin of (2,n)-torus links]
Let $L$ be a $(2,n)$-torus link (or knot) for some $n\in \mathbb{Z}-\{0,\pm 1\}$ and let $F=T(L)$. By lemmas \ref{lem:6bridge} and \ref{lem:bridget2}, every minimal bridge $(b;c_1,c_2,c_3)$-bridge trisection of $F$ is irreducible and satisfies $b=6$ and $c_i=2$. Thus, by Theorem \ref{thm_lower_L_c2}, $\Lcal(F)\geq 3 (6+1)=21$. We now discuss the opposite inequality.  
Figure \ref{fig:t2fig1} shows a picture of the three links $L_i=T_i \cup \T_{i+1}$ when $L$ is a Hopf link ($n=2$). 
For arbitrary $n$, the pants decompositions will be the same except for the long loops on each $P^i_{ij}$ which will resemble the Conway number of $L$ ($1/n$).  
The sequence of paths in $\Pcal(\Sigma)$ described in Figure \ref{fig:t2Lfig2} shows that $\Lcal(F)\leq 7+7+7=21$. 
Hence, $\Lcal(F)=21$.
\end{example}
\clearpage

\begin{figure}[ht!]
\includegraphics[width=16cm]{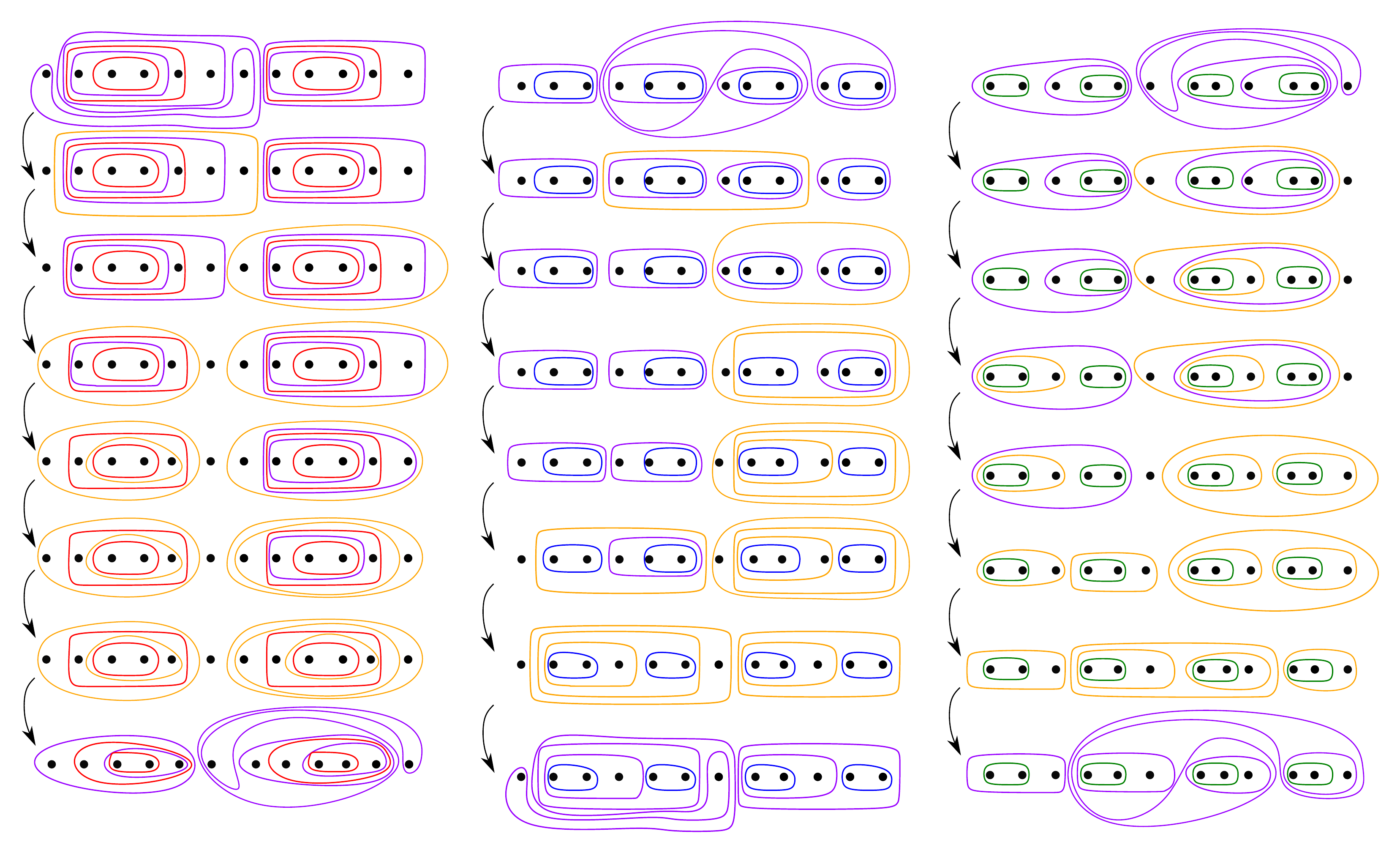}
\centering
\caption{$\Lcal$-invariant of the $T^2$-spin of a $(2,n)$-torus link. A path 
connecting $P^i_{ij}$ and $P^i_{ik}$. }
\label{fig:t2Lfig2}
\end{figure}




\section{Results in the relative setting}\label{section:relative}
We now turn to the relative setting. For more detailed definitions, readers can consult \cite{meier2020filling,blair2020kirby}. 
Let $Z$ be a 3-ball parametrized as $D^2\times I$. We call $\partial_{+}Z = D^2\times \lbrace 1 \rbrace$ the \textbf{positive boundary} and $\partial_{-}Z = D^2\times \lbrace -1 \rbrace$ the \textbf{negative boundary}. 
A \textbf{trivial relative tangle} $(Z,T)$ is a 3-ball $Z$ containing a properly embedded 1-manifold $T$ such that $\partial T$ is contained in the interior of $\partial_{-}Z \cup \partial_{+}Z$ such that there is a properly embedded arcs $\alpha\subset \partial_+Z$ such that $T$ can be isotoped, relative to $\partial T$, into $\alpha\times I$. 
A relative tangle is \textbf{strictly trivial} is $T$ is trivial and has no closed components and no components with both endpoints on $\partial_-Z$. A relative tangle is \textbf{spanning} if each arc component of $T$ has one endpoint on each $\partial_-Z$ and $\partial_+Z$. 
We will write \textbf{r.s.t. tangle} to refer to a relative spanning trivial tangle.  
Recall that a \textbf{trivial disk system} is a pair $(X,\mathcal{D}),$ where $X$ is the 4-ball, and $\mathcal{D}$ is a collection of properly embedded disks simultaneously isotopic, relative to the boundary, into $\partial X$. 


\begin{definition}
A \textbf{bridge trisection} of a properly embedded surface in the 4-ball $B^4$ is a decomposition $(B^4,F)=(X_1,\mathcal{D}_1) \cup (X_2,\mathcal{D}_2)\cup (X_3,\mathcal{D}_3)$, such that for 
$i\in \mathbb Z / 3\mathbb Z$,
\begin{enumerate}
\item $(X_i,\mathcal{D}_i)$ is a trivial disk system inside a 4-ball, 
\item $(Z,T_i)=(X_{i},\mc{D}_{i})\cap (X_{i+1},\mc{D}_{i+1})$ is a trivial relative tangle, and 
\item the triple intersection $(X_1,\mc{D}_1) \cap (X_2,\mc{D}_2) \cap (X_3,\mc{D}_3)$ is a $2b$ punctured 2-disk $\Sigma$. 
\end{enumerate}
\end{definition}
%
%

\textbf{Notation.} For a given bridge trisection of a properly embedded surface $F$ in $B^4$, we keep track of certain numerical invariants. Henceforth, we reserve these letters to mean the following:
\begin{itemize}
\item the \textbf{bridge number} of $\mc{T}$ is the quantity $b(\mc{T})=|\Sigma\cap F|/2$,
\item $b(F)$ denotes the minimal $b(\mc{T})$ among all bridge trisections of $F$, 
\item $b'$ denotes the number of bridge arcs in each strictly trivial relative tangle $(Z_{i},T_{i})$, 
\item $c_i$ denotes the number of closed components in each trivial spanning relative tangle $(Z_i,T_i)$, 
\item $v$ denotes the number of vertical arcs in each trivial spanning relative tangle $(Z_i,T_i)$.
\end{itemize}

We will refer to such $\mc T$ as a a $(b',c_1,c_2,c_3;v)$-bridge trisection. 
We say a properly embedded disk $D$ in a relative trivial tangle $(Z,T)$ is a \textbf{$c$-disk} if $\partial D \subset \partial_+Z$ is essential, int$(D)$ is disjoint from $\partial_+Z$, and $D$ intersects $T$ transversely in at most one point. 
An annulus $A$ properly embedded in $Z$ such that int$(A)$ disjoint from $T$ and with one boundary component a curve on each of $D^2 \times \{\pm 1\}$ is called a \textbf{spanning annulus}. A spanning annulus is \textbf{even} if it bounds a solid cylinder containing an even number of arcs. 
A collection of pairwise disjoint essential simple closed curves on a disk with punctures $\partial_-(Z,T) = \Sigma'$ is a \textbf{weak pants decomposition} if it is either empty and $\Sigma'$ has two or fewer punctures, or if it cuts $\Sigma'$ into pairs of pants and annuli, such that at most one of the pairs of pants is a once-punctured annulus and that annulus, if it exists, has $\partial \Sigma'$ as one of its boundary components. 
We say that a vertex $v \in C^*(\Sigma)$ or $\Pcal(\Sigma)$ belongs to the \textbf{disk set} $\mathcal{P}_c(T)$ if every curve of $v$ bounds a $c$-disk or an even spanning annulus in $(Z,T)$ and the boundaries of even spanning annuli in $\partial_-(Z,T)$ form a weak pants decomposition.
For a r.s.t. tangle $L$ in $B^3 = D^2\times I$ with bridge surface $\Sigma$ dividing $(B^3,L)$ into strictly trivial relative tangles $(Z_1,T_1)\cup_{\Sigma} (Z_2,T_2)$. 
A pair of vertices $u \in \mathcal{P}_c(T_1)$ and $w \in \mathcal{P}_c(T_2)$ is said to be an \textbf{efficient defining pair} if the distance $d(u,w)$ in $\C^*(\Sigma)$ is minimal.

\begin{definition}
Let $F$ be a properly embedded surface in $B^4$. If a bridge surface $\Sigma$ for $F$ is not admissible, then we define $\mathcal{L}^*(F) = 0$. 
For each $i<j$, let $\left(P^i_{ij}, P^j_{ij}\right)$ be an efficient pair for the r.s.t. tangle $T_i\cup \T_j$. 
We define $\mathcal{L}^*(\mathcal{T})$ of a bridge trisection $\mathcal{T}$ to be 
\[
d(P^1_{12},P^1_{31})+d(P^2_{23},P^2_{23})+d(P^3_{31},P^3_{23}),
\]
minimized over all choices of efficient defining pairs. The knotted surface invariant \textbf{$\mathcal{L}^*(F)$} is defined to be the minimum $\mathcal{L}^*(\mathcal{T})$ over all bridge trisections $\mathcal{T}$ for $F$ such that $b(\mathcal{T})=b(F)$.
\end{definition}

We restate Lemmas 5.8, 5.9 and 5.10 of \cite{blair2020kirby}, which we will often refer to, here for convenience. 
The proofs of these lemmas did not make use of the fact that the move involved is an edge in the pants complex, and so the lemmas also hold in the dual curve complex setting.
\begin{lemma}[Lemma 5.8 from \cite{blair2020kirby}]\label{lem:numberofreducingcompressions}
Suppose that $(Z,L)$ is a  
r.s.t. tangle equipped with a bridge splitting. Let $S$ be a collection of pairwise disjoint unpunctured reducing spheres or even reducing annuli for $\Sigma$ such that no two curves of $S \cap \Sigma$ bound an unpunctured annulus in $\Sigma-S$. Let $c$ be the number of closed components of $L$ and $v$ the number of arcs. Then, $2c + v - 3 + |\partial \Sigma| \geq |S|$.
\end{lemma}
\begin{lemma}[Lemma 5.9 and 5.10 from \cite{blair2020kirby}]
Suppose that $(Z,L)$ is either a spanning trivial relative tangle equipped with a bridge splitting with $b\geq 2$. Let $(u,w)$ be an efficient pair. If $d(u,w)>0$, then there exists a cut-reducing curve. Furthermore, any curve $\gamma$ in $v$ bounding a cut disk on one side is a curve in $w$ bounding a cut disk on the other side and $\gamma$ is a common curve in any geodesics from $u$ to $w.$ \label{lem:existencecommoncurve}
\end{lemma}

\begin{lemma}[Strengthened form of Lemma 5.7 of \cite{blair2020kirby}]\label{lem:efficientpairrelative}
Let $L$ be a spanning relative tangle with bridge number at least $b\geq 3/2$ equipped with a bridge splitting $(Z_1,T_1)\cup(Z_2,T_2)$. Take $u\in \mathcal{P}_c(T_1)$ and $w\in \mathcal{P}_c(T_2)$ such that $(u,w)$ is an efficient pair. Then, $d(u,w)=b-c+\frac{v}{2}$. Furthermore, each curve in $u$ moves at most once as we traverse a geodesic. 
\end{lemma}
\begin{proof}
In Lemma 5.7 of \cite{blair2020kirby}, the authors constructed an efficient pair with distance $d^P(u,w)=b-c+\frac{v}{2}$ so that each curve in $u$ moves at most once as we traverse a geodesic. Furthermore, they showed that the statements hold for all bridge trisections with $b \leq 3/2$. 

We induct on half integers $b\geq 2$. Suppose that the result holds for bridge splitting disk $\Sigma$ with bridge number $b'<b.$ Since we are in the spanning relative trivial tangle setting, there is a common curve $\gamma$. Such a common curve bounds cut disks or compressing disks or even spanning annuli on both sides. In other words, there is a reducing sphere, a connected sum sphere, or a reducing even annulus $S$ for our bridge splitting disk. We can surger $(B^3,L)$ along $S$ to get two bridge split spanning relative trivial tangles $(B^3,L_1) \cup (B^3,L_2)$ with bridge surfaces $\Sigma_1$ and $\Sigma_2$ (see Figure \ref{fig:surgery}). Again, if we let $u_i$ and $w_i$ denote the restriction of $v$ and $w$ to $\Sigma_i$, then we get that a geodesic $l$ connecting $u$ to $w$ restricts to a geodesic $l_i$ connecting $u_i$ to $w_i.$ Since the bridge number for $\Sigma_i$ is strictly less than $b$, 
we get that the length of $l_i$ is at least $b_i-c_i+\frac{v_i}{2}$ using the induction hypothesis. Therefore, $d(u,w)\geq b-c+\frac{v}{2}$.
\begin{figure}[ht!]
\centering
\includegraphics[width=1\textwidth]{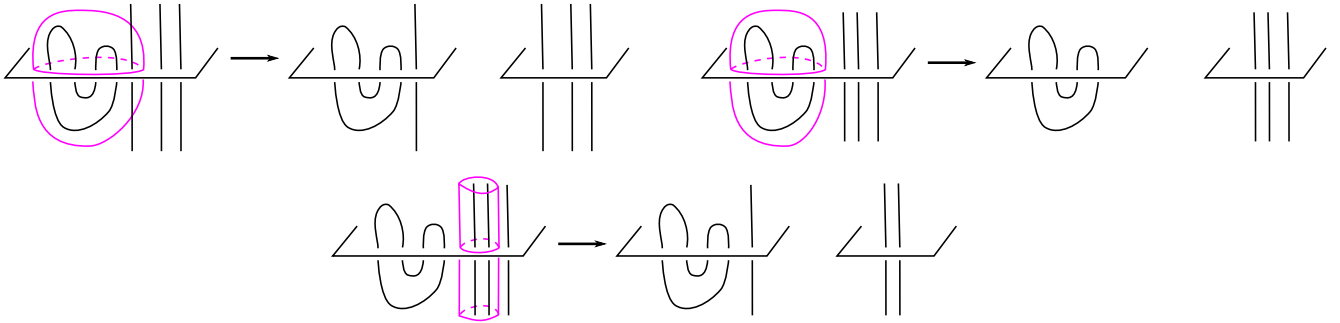}
\caption{Cutting the bridge splitting disk along a reducing sphere, a decomposing sphere, and a spanning even annulus.}
\label{fig:surgery}
\end{figure}
\end{proof}
By realizing that all cut-reducing curves must move, we get a slightly better bound than what was found in \cite{blair2020kirby}.
\begin{thm}\label{thm:lower_relcase}
Let $\mc T$ be an irreducible $(b',c_1,c_2,c_3;v)$-bridge trisection for a properly embedded surface $F$ in $B^4$. 
\begin{enumerate}[(i)]
\item If $2c_i +v \geq 2$ $\forall i$, then $\Lcal^*(\mc T) \geq 2b(\mc T) -\chi(F)$. 
\item If $2c_i +v < 2$ $\forall i$, then $\Lcal^*(\mc T) \geq 4b(\mc T) + \chi(F) -6$.
\end{enumerate}
\end{thm}
\begin{proof}
Let $(P^i_{i,i+1},P^{i+1}_{i,i+1})$ be efficient defining pairs for each tangle $L_i$, $i=1,2,3$. 
Let $\lambda$ be a geodesic in $C^*(\Sigma)$ connecting $P^i_{i,i+1}$ and $P^i_{i-1,i}$. 
Suppose there is a cut-curve $\gamma$ in $P^i_{i,i+1} \cap P^{i+1}_{i,i+1}$ that is fixed along $\lambda$. 
In particular, $\gamma\in P^i_{i-1,i}$ bounds a cut-disk in $(Z_{i-1},T_{i-1})$. By Lemma \ref{lem:existencecommoncurve}, the curve $\gamma$ bounds a cut disk in all three tangles in the spine. In other words, the bridge trisection decomposes as a connected sum, contradicting the fact that $\mc T$ is irreducible. 
Hence, $\lambda$ must move each cut-curve in $P^i_{i,i+1} \cap P^{i+1}_{i,i+1}$. \\
(i) There are $2b-2$ curves in each pants decomposition of $\Sigma$. By Lemma \ref{lem:efficientpairrelative}, the distance between $P^i_{i,i+1}$ and $P^{i+1}_{i,i+1}$ pair is $b - c_i + \frac{v}{2}$. This means that $|P^i_{i,i+1}\cap P^{i+1}_{i,i+1}|$ is at least $(2b-2)-(b-c_i-\frac{v}{2})=b+c_i-2+\frac{v}{2}$. Lemma \ref{lem:numberofreducingcompressions} implies that, among these common curves, at most $(2c_i+v -2)$ bound an even number of punctures. Thus, if $2c_i+v -2\geq 0$, then there are at least $(b+c_i-2+\frac{v}{2})-(2c_i+v -2)= b-c_i-\frac{v}{2}$ cut reducing curves in $P^i_{i,i+1}\cap P^{i+1}_{i,i+1}$. Hence, $d\left(P^i_{i,i+1}, P^i_{i-1,i}\right)\geq b-c_i-\frac{v}{2}$. 
By definition of $\Lcal^*(\mc T)$, we get that \[\mathcal{L}^*(\mc T)\geq (b-c_1-\frac{v}{2})+ (b-c_2-\frac{v}{2})+(b-c_3-\frac{v}{2}) = 3b-(c_1+c_2+c_3)-3\cdot \frac{v}{2}.\] 
Since $\chi(F) = c_1+c_2+c_3+v-b'$ and $2b=2b'+v$, we conclude that \[\mathcal{L}^*(\mc T)\geq 3\cdot \frac{2b'+v}{2}-(c_1+c_2+c_3)-3\cdot \frac{v}{2}=2b'+v-\chi(F)=2b-\chi(F).\]
(ii) On the other hand, if $2c_i+v<0$ for all $i$, then there are no even reducing curves in $P^i_{i,i+1}\cap P^{i+1}_{i,i+1}$. Thus, there are at least $b+c_i-2+\frac{v}{2}$ cut reducing curves in $P^i_{i,i+1}\cap P^{i+1}_{i,i+1}$. 
By definition of $\Lcal^*(\mc T)$ we conclude that, 
\[ \Lcal^*(\mc T) \geq 3b+c_1+c_2+c_3-6+3\cdot \frac{v}{2}=4b + \chi(F) -6.\]
\end{proof}

\begin{example}
Let $U\subset B^4$ be an unknotted 1-holed torus with unknotted boundary. We consider the bridge trisection $\mc{T}$ of $U$ built in Example 7.13 of \cite{meier2020filling}. 
The top row of Figure \ref{fig:unknottedtorus} shows the result of gluing pairwise trivial tangles of $\mc T$ together. In the bottom row, the red, blue, and green curves bound compressing disks in the red tangle, blue tangle, and green tangle, respectively. The gray curve bounds cut disks on two sides. To see that the distance in each disk set is one, the reader can pick two out of the three pictures in the second row and remove all simple closed curves that receive colors that are not common in the two pictures. The result is two pictures differing in only the gray curve. The two gray curves will differ by an edge in the dual curve complex. Therefore $\mathcal{L}(\mc T) \leq 3$. 
On the other hand, Theorem \ref{thm:lower_relcase} implies that $\Lcal^*(\mc T)\geq 4\cdot \frac{5}{2} + (-1) -6=3$. Therefore, $\Lcal^*(\mc T) =\Lcal(\mc T)=3$. 
\end{example}
\clearpage
\begin{figure}[ht!]
\includegraphics[width=17cm]{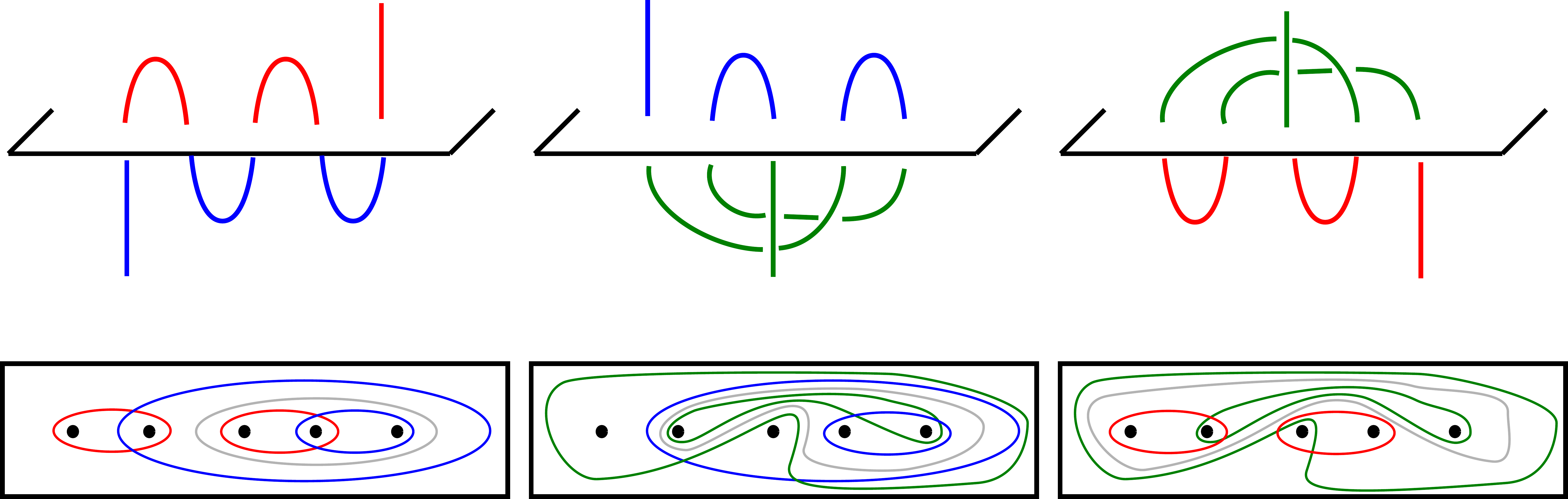}
\caption{A $b=2$ bridge trisection for the 1-holed unknotted torus $U$. The second row shows efficient defining pairs for each r.s.t. tangle $L_i$.}
\centering\label{fig:unknottedtorus}
\end{figure}


\subsection*{Estimating $\Lcal^*$-invariant of a ribbon disk}


%
%
Let $F$ be the standard ribbon disk for the square knot. The goal of this subsection is to estimate $\Lcal^*(F)$. Recall that $\mathcal{L}$ is defined for relative trisections realizing the bridge number. 

\begin{proposition}\label{prop:bridge_of_disk_minimal}
The minimum number of punctures that a bridge trisection disk for $F$ can have is seven.
\end{proposition}
\begin{proof}
 Since the square knot has braid index 3, we know that $b(F)\geq 3/2.$ Also,   $\chi(F)=c_1+c_2+c_3+v-b'$, which means that the minimum number of bridge arcs we can have is $b'\geq 2+\min c_1+\min c_2+\min c_3.$ Meier in \cite{meier2020filling}, produced a diagram showing that $\min c_i=0$ for $i=1,2,3$. Each bridge arc contributes two punctures to the bridge trisection disk, and therefore, the bridge number is at least $7/2$. The top half of Figure \ref{fig:slicedisk}, which also appeared in \cite{meier2020filling} shows that the bridge number is at most 7/2 and the equality follows.
\end{proof}

\begin{proposition}
Every minimal bridge trisection of the standard ribbon disk $F$ for \[T(2,2k-1)\# -T(2,2k-1)\] is irreducible.\label{prop:ribbon_irred}
\end{proposition}
\begin{proof}
Suppose there is a simple closed curve $c$ on the bridge trisection disk $\Sigma$ bounding a compression $D_i$ in $(Z_i,T_i)$. Since $c_i=0$ for a minimal bridge trisection of $F$, the compression $D_i$ cannot be a compressing disk. Also, $D_i$ cannot be an even spanning annulus because if so, $F$ is disconnected. Therefore, the curve $c$ bounds a cut disk $D_i$ in $(Z_i,T_i)$. It follows that $F$ decomposes as a connected sum of two surfaces $F_1\# F_2$. Without loss of generality, $F_1$ is a closed surface admitting a bridge trisection with $b=2$ or $b=3.$ Since $F$ is a disk, the surface $F_1$ is a sphere. But a knotted sphere with bridge number 2 or 3 is stabilized. This means that the bridge number of $F$ is not minimal, which is a contradiction.
\end{proof}
\begin{example}
Let $F\subset B^4$ be the standard ribbon disk for the square knot. We consider the bridge trisection $\mc{T}$ of $F$ built in Example 3.15 of \cite{meier2020filling}. Each image in Figure \ref{fig_rel:EfficientDefiningPairs} depicts an efficient defining pair $(P_{ij}^i,P_{ij}^j)$. Recall that for the 7 times-punctured disk, we have that the distance $d(P_{ij}^i,P_{ij}^j) =2$. Here, a curve and its image under the A-move are drawn concurrently as two curves intersecting in two points. Each of the red, blue and green curves is the boundary of a regular neighborhood of a bridge shadow, and hence bounds a compressing disk on one side. 

Each purple curve in Figure \ref{fig_rel:EfficientDefiningPairs} bounds spanning annulus on both sides due to the following reason. On each side, it is easier to see that each purple curve bounds a disk $\Delta$ punctured twice by two vertical arcs $v_1,v_2$. Then, one verifies that it is possible to find a embedded arc $a$ on $\Sigma$ connecting the two punctures on $\Sigma$ associated to the $v_1$ and $v_2$. There is then a disk $a\times [-1,1]$ whose boundary is $v_1 \cup v_2 \cup a\times\{-1\} \cup a\times\{1\}$ that allows us to perform surgery to turn a twice punctured disk $\Delta$ into an annulus. After that, the boundaries of the annulus can be extended to lie precisely in the positive and the negative boundaries. 

Each yellow curve in Figure \ref{fig_rel:EfficientDefiningPairs} bounds a cut disk on two sides due to the following reasons. Each relative spanning tangle $(Z_i,T_i)\cup(Z_j,T_j)$ admits a system of cancelling disks, which are bridge disks $\lbrace D_1,\cdots, D_k\rbrace$ ($k=2$ or 4) such that 1) $D_i$ is on opposite side of $D_{i+1}$, 2) $D_i\cap D_{i+1}$ in precisely one point on the relative spanning tangle, and 3) $D_i\cap D_{j}=\emptyset$ for non-consecutive $i,j$. A yellow curve in Figure \ref{fig_rel:EfficientDefiningPairs} is either the boundary of a regular neighborhood of shadows of $D_1\cup D_2$ (cuts off a thrice-punctured disk from $\Sigma$) or the boundary of a regular neighborhood of shadows of $D_1\cup D_2\cup D_3\cup D_4$ (cuts off a 5 times-punctured disk from $\Sigma$). In either case, such a curve bounds a cut disk on two sides.

Figure \ref{fig_rel_red} depicts a path connecting $P_{12}^1$ to $P_{31}^1$ of length 4. Figure \ref{fig_rel_blue} depicts a path connecting $P_{12}^2$ to $P_{23}^2$ of length 3. Figure \ref{fig_rel_green} depicts a a path connecting $P_{31}^3$ to $P_{23}^3$ of length 9. By Proposition \ref{prop:bridge_of_disk_minimal} and Proposition \ref{prop:ribbon_irred}, any minimal bridge trisection of $F$ is irreducible. This allows us to apply Theorem \ref{thm:lower_relcase}, to conclude that $6\leq\mc{L}^*(F)\leq 4+3+9=16$.
\end{example}
\begin{figure}[ht!]
\centering
\includegraphics[width=7cm]{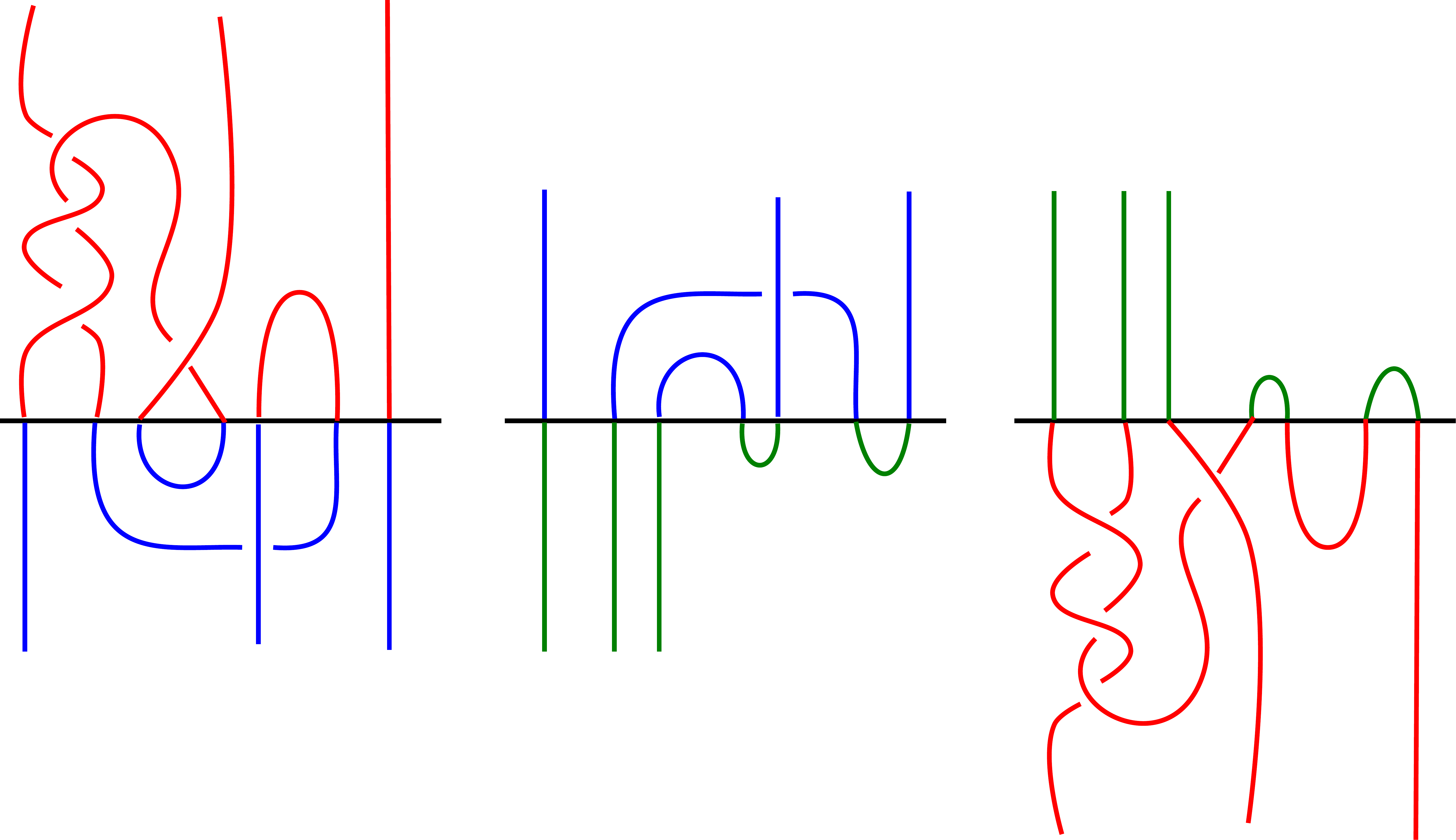}
\caption{The three spanning trivial tangles appearing in a bridge trisected slice disk for the square knot.}
\label{fig:slicedisk}
\end{figure}
%
%
\begin{figure}[ht!]
\centering
\includegraphics[width=15cm]{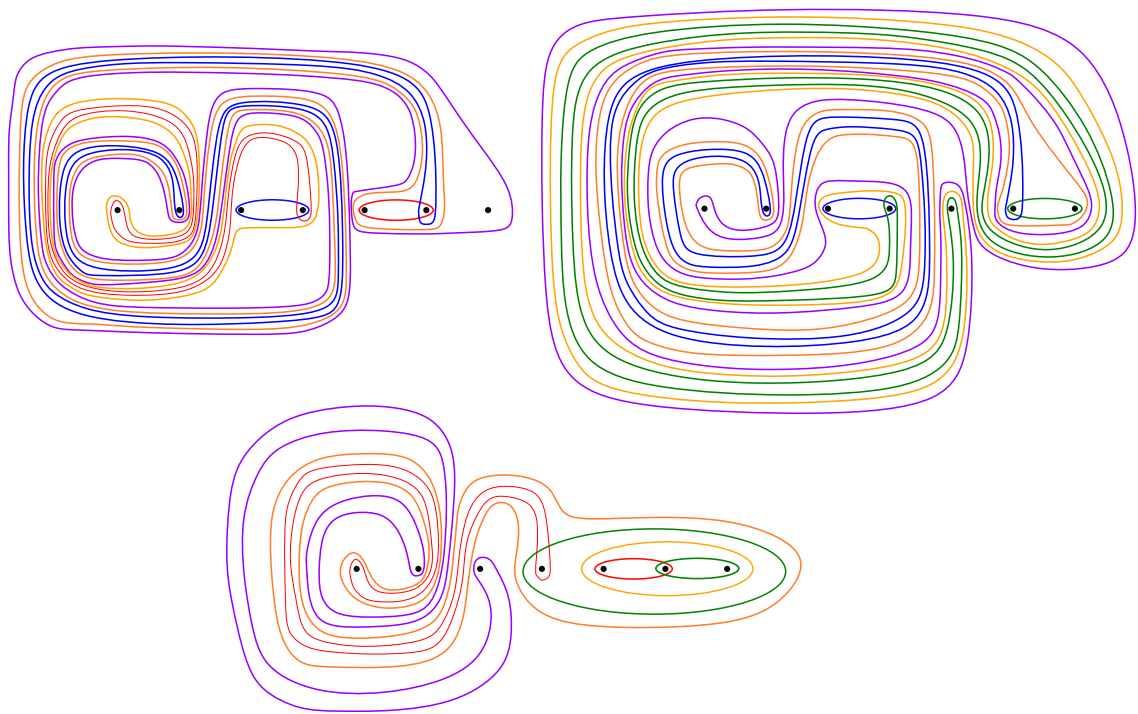}
\caption{Efficient defining pairs.}
\label{fig_rel:EfficientDefiningPairs}
\end{figure}
\begin{figure}[ht!]
\centering
\includegraphics[width=15cm]{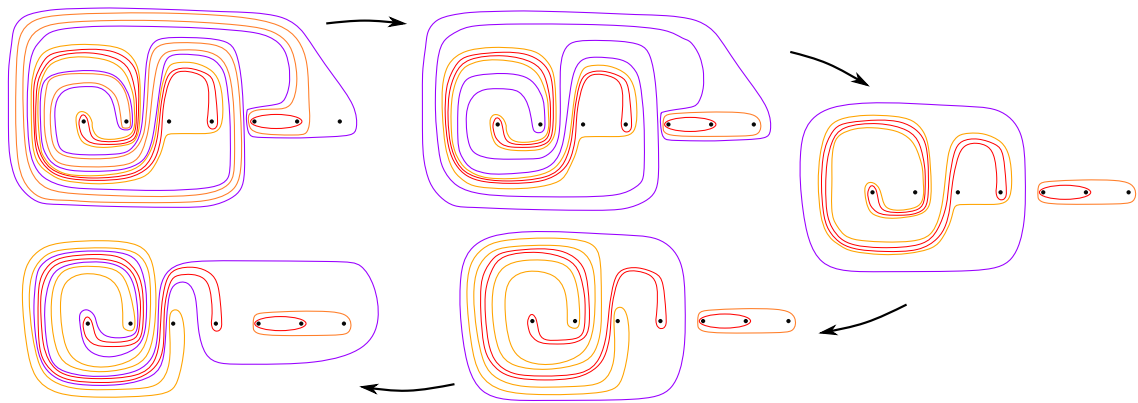}
\caption{A path connecting $P_{12}^1$ and $P_{31}^1$.}
\centering\label{fig_rel_red}
\end{figure}
\begin{figure}[ht!]
\centering
\includegraphics[width=13cm]{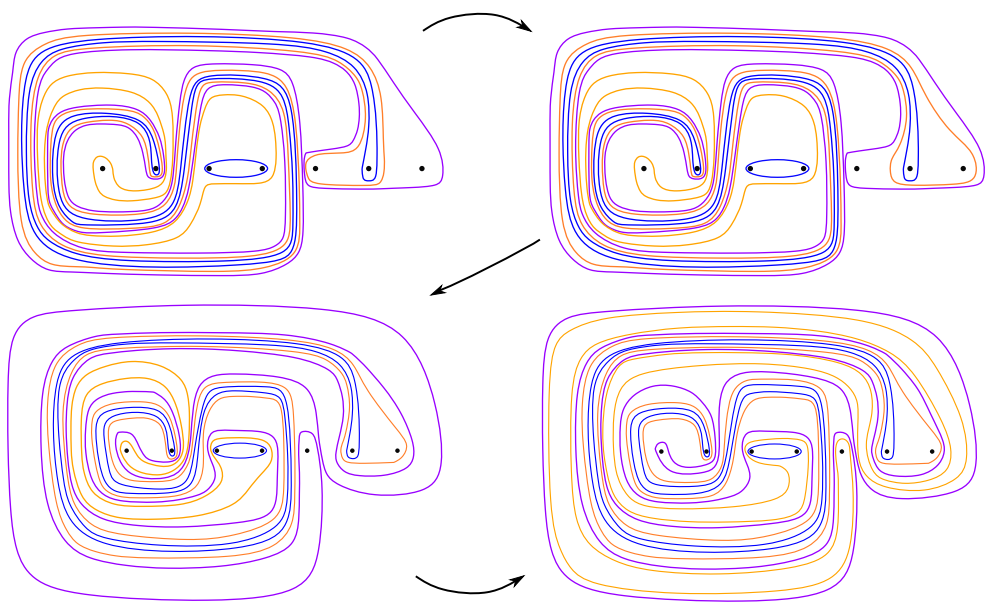}
\caption{A path connecting $P_{12}^2$ and $P_{23}^2$.}
\centering\label{fig_rel_blue}
\end{figure}
\begin{figure}[ht!]
\centering
\includegraphics[width=15cm]{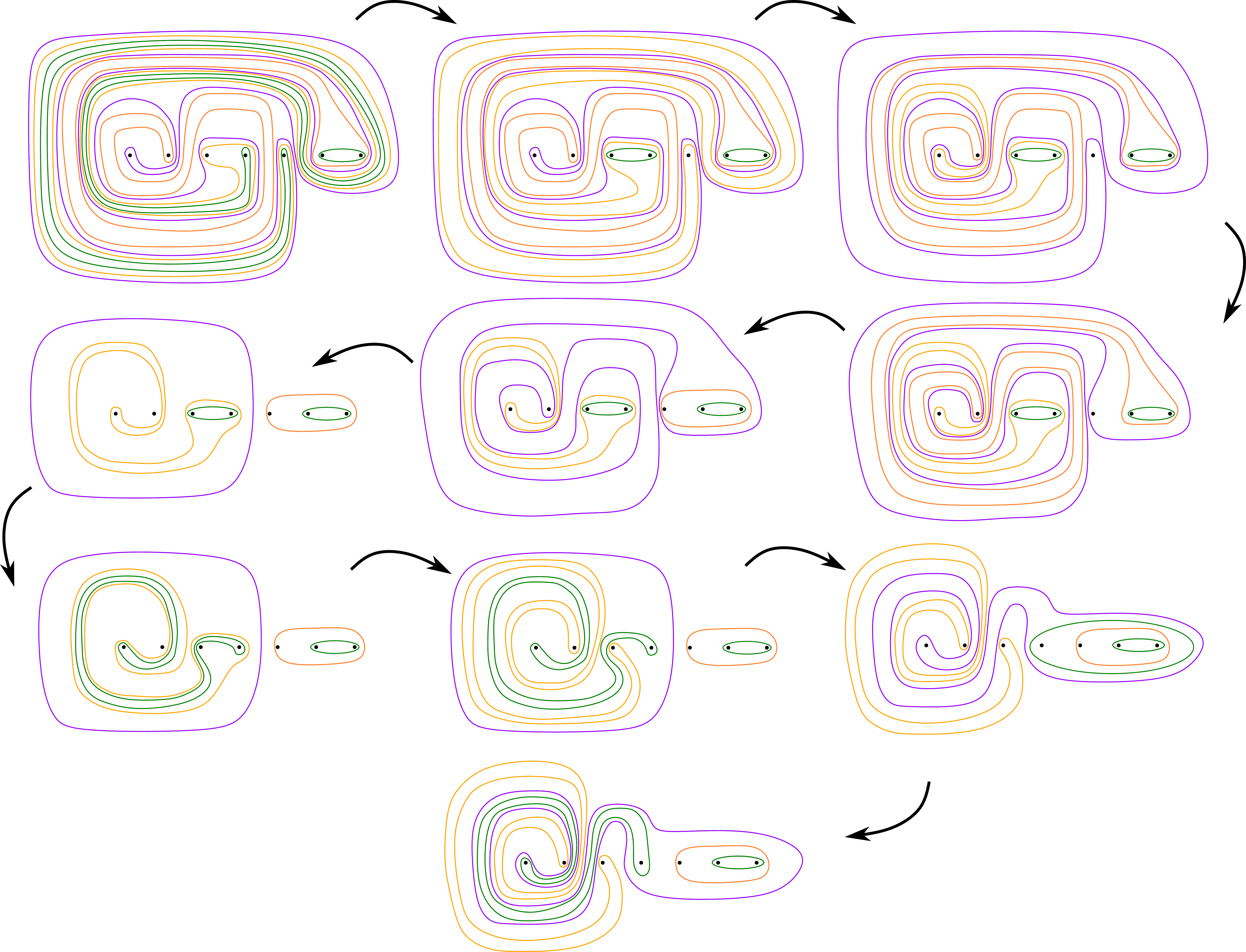}
\caption{A path connecting $P_{31}^3$ and $P_{23}^3$.}
\centering\label{fig_rel_green}
\end{figure}



\clearpage

\bibliographystyle{plain}
{ \small
\bibliography{main}

\begin{thebibliography}{10}

\bibitem{aranda2021bounding}
Rom{\'a}n Aranda, Puttipong Pongtanapaisan, Scott~A Taylor, and Suixin Zhang.
\newblock Bounding the kirby-thompson invariant of spun knots.
\newblock {\em arXiv preprint arXiv:2112.02420}, 2021.

\bibitem{artin1925isotopie}
Emil Artin.
\newblock Zur isotopie zweidimensionaler fl{\"a}chen imr4.
\newblock {\em Abhandlungen aus dem Mathematischen Seminar der Universit{\"a}t
  Hamburg}, 4:174--177, 1925.

\bibitem{bachman2005distance}
David Bachman and Saul Schleimer.
\newblock Distance and bridge position.
\newblock {\em Pacific journal of mathematics}, 219(2):221--235, 2005.

\bibitem{blair2020kirby}
Ryan Blair, Marion Campisi, Scott~A Taylor, and Maggy Tomova.
\newblock Kirby-thompson distance for trisections of knotted surfaces.
\newblock {\em Journal of the London Mathematical Society}, 2021.

\bibitem{hatcher_thurston}
A.~Hatcher and W.~Thurston.
\newblock Incompressible surfaces in {$2$}-bridge knot complements.
\newblock {\em Invent. Math.}, 79(2):225--246, 1985.

\bibitem{hayashi1998heegaard}
Chuichiro Hayashi and Koya Shimokawa.
\newblock Heegaard splittings of the trivial knot.
\newblock {\em J. Knot Theory Ramifications}, 7(8):1073--1085, 1998.

\bibitem{joseph2021bridge}
Jason Joseph and Puttipong Pongtanapaisan.
\newblock Bridge numbers and meridional ranks of knotted surfaces and welded
  knots.
\newblock {\em arXiv preprint arXiv:2111.09233}, 2021.

\bibitem{kirby2018new}
Robion Kirby and Abigail Thompson.
\newblock A new invariant of 4-manifolds.
\newblock {\em Proceedings of the National Academy of Sciences},
  115(43):10857--10860, 2018.

\bibitem{livingston1985stably}
Charles Livingston.
\newblock Stably irreducible surfaces in s4.
\newblock {\em Pacific Journal of Mathematics}, 116(1):77--84, 1985.

\bibitem{meier2020filling}
Jeffrey Meier.
\newblock Filling braided links with trisected surfaces.
\newblock {\em arXiv preprint arXiv:2010.11135}, 2020.

\bibitem{meier2017bridge}
Jeffrey Meier and Alexander Zupan.
\newblock Bridge trisections of knotted surfaces in {$S^4$}.
\newblock {\em Transactions of the American Mathematical Society},
  369(10):7343--7386, 2017.

\bibitem{ogawa2021trisections}
Masaki Ogawa.
\newblock Trisections with kirby-thompson length 2.
\newblock {\em arXiv preprint arXiv:2112.10033}, 2021.

\bibitem{otal1982perturbed}
Jean-Pierre Otal.
\newblock Pr\'{e}sentations en ponts du n\oe ud trivial.
\newblock {\em C. R. Acad. Sci. Paris S\'{e}r. I Math.}, 294(16):553--556,
  1982.

\bibitem{yoshikawa1994enumeration}
Katsuyuki Yoshikawa.
\newblock An enumeration of surfaces in four-space.
\newblock {\em Osaka Journal of Mathematics}, 31(3):497--522, 1994.

\bibitem{zupan2013bridge}
Alexander Zupan.
\newblock Bridge and pants complexities of knots.
\newblock {\em Journal of the London Mathematical Society}, 87(1):43--68, 2013.

\end{thebibliography}
}

$\quad$ \\
Rom\'an Aranda, Binghamton University 
\hfill 
Suixin (Cindy) Zhang, Colby College\\
email: \texttt{jaranda@binghamton.edu} 
\hfill
email: \texttt{szhang22@colby.edu}\\
$\quad$ \\
Puttipong Pongtanapaisan, University of Saskatchewan\\
email: \texttt{puttipong@usask.ca}
\end{document}